\documentclass[UTK8]{ctexart}
\usepackage{authblk}
\usepackage{indentfirst}%
\usepackage{upgreek}
\usepackage{amssymb}
\usepackage{amsmath}
\usepackage{mathrsfs}
\usepackage{geometry}
\usepackage{amsthm}
\usepackage{enumerate}
\usepackage{graphicx}
\usepackage{float}
\usepackage{subfigure}
\usepackage{tikz}
\usepackage{setspace}
\usepackage{multicol}
\usepackage{caption}
\usepackage[english]{babel}
\newtheorem{proposition}{Proposition}[section]

\newtheorem{thm}{\bf Theorem~}[section]

\newtheorem{corollary}[thm]{\bf Corllary~}
\newtheorem{lemma}[thm]{\bf Lemma}
\newtheorem{remark}{\bf Remark~}[section]

\newtheorem{theoremB}{Theorem}

\usepackage{titlesec}
\titleformat{\section}
{\normalfont\Large\bfseries}{\thesection.}{0.4em}{}
\titleformat{\subsection}
{\normalfont\large\bfseries}{\thesubsection.}{0.4em}{}
\titleformat{\subsubsection}
{\normalfont\normalsize\bfseries}{\thesubsubsection.}{0.4em}{}

\begin{document}
	
	\newpage

		\title{A variational approach to periodic orbits in the  $e^{-}Z^{2+}e^{-}$ Helium}
		\date{}
		%\author[a]{Urs Frauenfelder}
		\author{Zixuan Ye\thanks{%
					Zixuan Ye\\%
					\texttt{2250501033@cnu.edu.cn}\\%
					School of Mathematical Sciences,\ Capital Normal University,\ Beijing 100048 P.R. China%
				}}
		%\affil[a]{Institute of Mathematics,\ Augsburg University}
		%\affil{School of Mathematical Sciences,\ Capital Normal University,\ Beijing 100048 P.R. China}
	\maketitle
	%Author
	
	%Abstract
	\begin{abstract}
		{In this article, we use variational approaches to describe generalized solutions $(q_1,q_2)$ and critical points $(z_1,z_2)$ of the action functional $\mathscr{B}_{av}$ for the Helium atom in the $e^{-}Z^{2+}e^{-}$ configuration with mean interaction, where $(q_1,q_2)$ and $(z_1,z_2)$ are related by a non-local Levi-Civita regularization introduced by Barutello, Ortega and Verzini. Additionally, we give the Lagrangian and the Hamiltonian formulations of the generalized solutions $(q_1,q_2)$ following the framework constructed by Cieliebak, Frauenfelder and Volkov. Finally, we count the number of periodic orbits $(z_1,z_2)\in \mathscr{C}_{\mathscr{B}_{av}}$  and find the 1-to-1 correspondence between them and positive rational numbers $\mathbb{Q}_{+}$. \rm}  
		
	\end{abstract}
	%Content
	\tableofcontents
	\section{Introduction}
%	Describe the motion of frozen planet and $e^{-1}Z^{+2}e^{-1}$ configuration in the language of Lagrange and Hamilton both mathematically and physically.
	
%	References are \cite{barutello2021regularized}, \cite{cieliebak2022variational} and \cite{wintgen1992semiclassical}.
	The classical dynamics for the Helium atom is of mixed phase space, where regular and chaotic motions coexist. A proper semiclassical treatment of the Helium atom is highly desirable, and the periodic orbits in the Helium atom play an important role in its semiclassical treatment. \cite{tanner2000theory} and \cite{wintgen1992semiclassical} both emphasize the importance of studying periodic orbits in Helium.
	
	In this article, we will focus exclusively on near-collinear configurations. There are two near-collinear configurations in the Helium atom: the frozen planet configuration and the ${e^{-}}Z^{2+}e^{-}$ configuration.

	 In the situation of frozen planet configuration, two electrons are located on the same side of the nucleus. \cite{cieliebak2022variational} and \cite{cieliebak2023nondegeneracy} explore a variational approach to this case, \cite{frauenfelder2021helium} proposes considering only the mean interaction between the two electrons. And \cite{frauenfelder2021helium} gives a way to construct Floer homology for frozen planet orbits, as further detailed in \cite{frauenfelder2023compactness}. More recently, Baranzini, Canneori and Terracini(\cite{baranzini2025frozen}) investigate frozen planet orbits for the $n$-electron atom, they demonstrate that frozen planet orbits converge to segments of a brake orbit in a Kepler-type problem, and thereby establishing a strong analogy with Schubart orbits in the gravitational n-body problem. 
	 
	 On the other hand, it's natural to be curious about the situation when these two electrons lie on the opposite sides of the nucleus. We refer to the Helium atom in this configuration as the ${e^{-}}Z^{2+}e^{-}$ configuration. Unlike the frozen planet configuration, if the two electrons are confined to a line, this system is not integrable and fully chaotic. The left electron and the right one both undergo consecutive collisions with the nucleus, and their motions resemble free-falls, as shown below.

	 
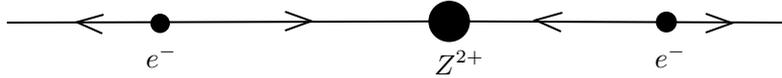
\begin{figure}[H]
	 	\centering

	 	\tikzset{every picture/.style={line width=0.75pt}} %set default line width to 0.75pt        
	 	
	 	\begin{tikzpicture}[x=0.75pt,y=0.75pt,yscale=-1,xscale=1]
	 		%uncomment if require: \path (0,310); %set diagram left start at 0, and has height of 310
	 		
	 		%Shape: Circle [id:dp24136711802427757] 
	 		\draw  [fill={rgb, 255:red, 0; green, 0; blue, 0 }  ,fill opacity=1 ] (257.99,193.88) .. controls (257.99,188.42) and (262.41,184) .. (267.87,184) .. controls (273.32,184) and (277.74,188.42) .. (277.74,193.88) .. controls (277.74,199.33) and (273.32,203.75) .. (267.87,203.75) .. controls (262.41,203.75) and (257.99,199.33) .. (257.99,193.88) -- cycle ;
	 		%Shape: Circle [id:dp4795244372346217] 
	 		\draw  [fill={rgb, 255:red, 0; green, 0; blue, 0 }  ,fill opacity=1 ] (371.52,194.18) .. controls (371.52,191.56) and (373.65,189.44) .. (376.27,189.44) .. controls (378.89,189.44) and (381.01,191.56) .. (381.01,194.18) .. controls (381.01,196.8) and (378.89,198.93) .. (376.27,198.93) .. controls (373.65,198.93) and (371.52,196.8) .. (371.52,194.18) -- cycle ;
	 		%Straight Lines [id:da5896541858011335] 
	 		\draw    (277.74,193.88) -- (438.54,194.49) ;
	 		\draw   (323.64,198.95) -- (310.94,193.27) -- (324.33,189.49) ;
	 		%Shape: Circle [id:dp43467266835702345] 
	 		\draw  [fill={rgb, 255:red, 0; green, 0; blue, 0 }  ,fill opacity=1 ] (119.18,194.86) .. controls (119.18,192.45) and (121.13,190.49) .. (123.55,190.49) .. controls (125.96,190.49) and (127.92,192.45) .. (127.92,194.86) .. controls (127.92,197.27) and (125.96,199.23) .. (123.55,199.23) .. controls (121.13,199.23) and (119.18,197.27) .. (119.18,194.86) -- cycle ;
	 		%Straight Lines [id:da4679607743033958] 
	 		\draw    (47.04,194.49) -- (257.99,193.88) ;
	 		\draw   (185.92,189) -- (199,193.74) -- (185.92,198.49) ;
	 		\draw   (94.64,199.95) -- (81.94,194.27) -- (95.33,190.49) ;
	 		\draw   (395.92,190) -- (409,194.74) -- (395.92,199.49) ;

	 		% Text Node
	 		\draw (115,204) node [anchor=north west][inner sep=0.75pt]   [align=left] {$\displaystyle e^{-}$};
	 		% Text Node
	 		\draw (369,204) node [anchor=north west][inner sep=0.75pt]   [align=left] {$\displaystyle e^{-}$};
	 		% Text Node
	 		\draw (259,207) node [anchor=north west][inner sep=0.75pt]   [align=left] {$\displaystyle Z^{2+}$};

	 	\end{tikzpicture}

	 \end{figure}
 
    In \cite{wintgen1992semiclassical}, Wintgen, Richter and Tanner use a numerical method to detect the collinear motion of the Helium atom, finding that the periodic orbits obey a binary coding, and they characterize the motion on the potential surface mostly. \cite{zhao2023shooting} uses shooting method to find periodic orbits both in the frozen planet case and the ${e^{-}}Z^{2+}e^{-}$ case.
    
	In this article, we first study the generalized solutions $(q_1,q_2)$ in the ${e^{-}}Z^{2+}e^{-}$ configuration with mean interaction from the perspective of variational methods. We use a non-local regularization established in \cite{barutello2021regularized} by Barutello, Ortega and Verzini, and there is a torus action that naturally arises from this reparameterization. After regularizing and modding out this torus action, we study the action functional $\mathscr{B}_{av}$ and its critical points  $(z_1,z_2)$, i.e. $\mathscr{C}_{\mathscr{B}_{av}}$. 
	
	Using the non-local regularization in \cite{barutello2021regularized}, we first get a 2-to-1 surjective map \begin{equation*}
		\begin{aligned}
			\Phi_{LC}:& \left\{ \left. z\in {{C}^{0}}\left( {{S}^{1}},\mathbb{R} \right) \right|z\text{ has precisely }m\text{ zeros and switches sign at each zero} \right\} \\ 
			& \to \left\{ \left. q\in {{C}^{0}}\left( {{S}^{1}},{{\mathbb{R}}_{\ge 0}} \right) \right|q\text{ has precisely }m\text{ zeros and }\int_{0}^{1}{\frac{ds}{q\left( s \right)}<\infty } \right\}, \\ 
		\end{aligned}
	\end{equation*}
here $m\in 2\mathbb{N}$.

 Then we prove a 4-to-1 relationship between critical points $(z_1,z_2)$ of the action functional $\mathscr{B}_{av}$ and generalized solutions $(q_1,q_2)$ of the mean interaction equation. More precisely:

	\begin{theoremB}\label{Main1}
		Under the Levi-Civita transformations $q_i(t)=z_i{(\tau_i(t))}^2(i=1,2)$ with time change $\tau_i(t)$ satisfying $\tau_i(0)=0$ and $\frac{dt}{q_i(t)}=\frac{d\tau_i(t)}{\parallel z_i\parallel ^2}$, critical points $(z_1,z_2)$ of the action functional $\mathscr{B}_{av}$ are in 4-to-1 correspondence with generalized solutions $(q_1,q_2)$ of equation \begin{equation*}
			\begin{aligned}
				&\ddot{q_1}=-\frac{2}{\left(q_1\right)^2}+\frac{1}{ \left(\bar{q}_{1}+\bar{q}_{2}\right)^2},\\
				&\ddot{q_2}=-\frac{2}{\left(q_2\right)^2}+\frac{1}{ \left(\bar{q}_{1}+\bar{q}_{2}\right)^2}.
			\end{aligned}
		\end{equation*}.
	\end{theoremB} 
	 
	 For any two relatively prime positive integers $n_1$ and $n_2$, we can choose $\sigma_i=\frac{1}{2n_i}(i=1,2)$. Given any $\sigma_{1}>0$, $\sigma_{2}>0$, we employ analytic techniques to obtain a unique solution $(q_1^{\sigma_{1}},q_2^{\sigma_{2}})$ for the mean interaction equation of Helium, where $\sigma_{i}$ are the period of $q_i$, $i=1,2$. Then we extend $(q_1^{\sigma_{1}},q_2^{\sigma_{2}})$ to solutions of period 1, i.e. critical points $(z_1,z_2)$ of the action functional $\mathscr{B}_{av}$. One can show that $\frac{\sigma_{1}}{\sigma_{2}}$ is a strictly monotonicity increasing bijection from $\mathbb{R}_{\ge 0}$ to $\mathbb{R}_{\ge 0}$ through analytic computations. Based on Theorem \ref{Main1}, we get a one-to-one correspondence between $(z_1,z_2)\in\mathscr{C}_{\mathscr{B}_{av}}$ and the set of positive rational numbers $\mathbb{Q}_{+}$, which constitutes the main result of this paper. 
	 
	 \begin{theoremB}\label{Main2}
	 		There is a one-to-one correspondence between $\mathbb{Q}_{+}$ and $\mathscr{C}_{\mathscr{B}_{av}}/{\mathbb{Z}}/{2\mathbb{Z}}\;\times {\mathbb{Z}}/{2\mathbb{Z}}\;$, in the sense that $z_i=-z_i(i=1,2)$ under the  ${\mathbb{Z}}/{2\mathbb{Z}}\;$ action.
	 \end{theoremB}
	 
	 This theorem shows that $\forall r=\frac{n_2}{n_1}\in \mathbb{Q}_{+}$, $gcd(n_1,n_2)=1$, we can get a unique periodic orbit $(z_1,z_2)$. And in this situation,the theorem states that the left electron and the right one collide with the nucleus $n_2$ and $n_1$ times in period 1 separately. This one-to-one correspondence shows us the motion of the uncoupled system  \begin{equation*}
	 	\begin{aligned}
	 		& {{z}''_{1}}={{a}_{1}}{{z}_{1}}+{{b}_{1}}z_{1}^{3},\\ 
	 		& {{z}''_{2}}={{a}_{2}}{{z}_{2}}+{{b}_{2}}z_{2}^{3},\\ 
	 	\end{aligned} 
	 \end{equation*}  as following.
	 	
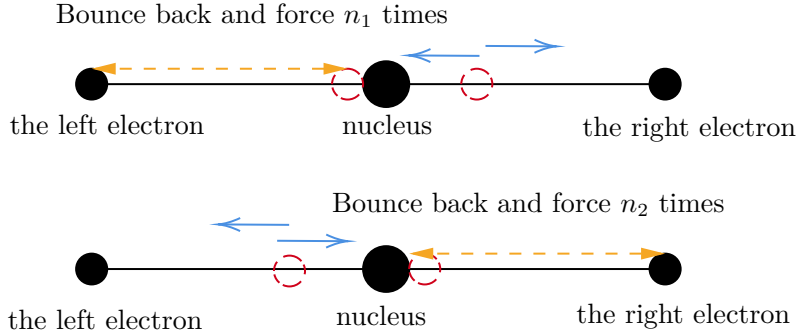
\begin{figure}[H]
	 		\centering
	 		
	 		\tikzset{every picture/.style={line width=0.75pt}} %set default line width to 0.75pt        
	 		
	 		\begin{tikzpicture}[x=0.75pt,y=0.75pt,yscale=-1,xscale=1]
	 			
	 			%Straight Lines [id:da8622703913324349] 
	 			\draw    (66.41,82.22) -- (360.41,82.22) ;
	 			%Shape: Circle [id:dp913754525460246] 
	 			\draw  [fill={rgb, 255:red, 0; green, 0; blue, 0 }  ,fill opacity=1 ] (201.91,82.22) .. controls (201.91,75.87) and (207.06,70.72) .. (213.41,70.72) .. controls (219.76,70.72) and (224.91,75.87) .. (224.91,82.22) .. controls (224.91,88.57) and (219.76,93.72) .. (213.41,93.72) .. controls (207.06,93.72) and (201.91,88.57) .. (201.91,82.22) -- cycle ;
	 			%Shape: Circle [id:dp21263899434143052] 
	 			\draw  [fill={rgb, 255:red, 0; green, 0; blue, 0 }  ,fill opacity=1 ] (58.66,82.22) .. controls (58.66,77.94) and (62.13,74.47) .. (66.41,74.47) .. controls (70.69,74.47) and (74.16,77.94) .. (74.16,82.22) .. controls (74.16,86.5) and (70.69,89.97) .. (66.41,89.97) .. controls (62.13,89.97) and (58.66,86.5) .. (58.66,82.22) -- cycle ;
	 			%Shape: Circle [id:dp2627749488418015] 
	 			\draw  [fill={rgb, 255:red, 0; green, 0; blue, 0 }  ,fill opacity=1 ] (344.91,82.22) .. controls (344.91,77.94) and (348.38,74.47) .. (352.66,74.47) .. controls (356.94,74.47) and (360.41,77.94) .. (360.41,82.22) .. controls (360.41,86.5) and (356.94,89.97) .. (352.66,89.97) .. controls (348.38,89.97) and (344.91,86.5) .. (344.91,82.22) -- cycle ;
	 			%Shape: Circle [id:dp14087843301856906] 
	 			\draw  [color={rgb, 255:red, 208; green, 2; blue, 27 }  ,draw opacity=1 ][dash pattern={on 3.75pt off 3pt on 7.5pt off 1.5pt}] (186.41,82.22) .. controls (186.41,77.94) and (189.88,74.47) .. (194.16,74.47) .. controls (198.44,74.47) and (201.91,77.94) .. (201.91,82.22) .. controls (201.91,86.5) and (198.44,89.97) .. (194.16,89.97) .. controls (189.88,89.97) and (186.41,86.5) .. (186.41,82.22) -- cycle ;
	 			%Shape: Circle [id:dp15020063643220927] 
	 			\draw  [color={rgb, 255:red, 208; green, 2; blue, 27 }  ,draw opacity=1 ][dash pattern={on 3.75pt off 3pt on 7.5pt off 1.5pt}] (250.91,82.22) .. controls (250.91,77.94) and (254.38,74.47) .. (258.66,74.47) .. controls (262.94,74.47) and (266.41,77.94) .. (266.41,82.22) .. controls (266.41,86.5) and (262.94,89.97) .. (258.66,89.97) .. controls (254.38,89.97) and (250.91,86.5) .. (250.91,82.22) -- cycle ;
	 			%Straight Lines [id:da23862712805495367] 
	 			\draw    (66.41,175.22) -- (360.41,175.22) ;
	 			%Shape: Circle [id:dp11776978052386589] 
	 			\draw  [fill={rgb, 255:red, 0; green, 0; blue, 0 }  ,fill opacity=1 ] (201.91,175.22) .. controls (201.91,168.87) and (207.06,163.72) .. (213.41,163.72) .. controls (219.76,163.72) and (224.91,168.87) .. (224.91,175.22) .. controls (224.91,181.57) and (219.76,186.72) .. (213.41,186.72) .. controls (207.06,186.72) and (201.91,181.57) .. (201.91,175.22) -- cycle ;
	 			%Shape: Circle [id:dp5712218653431705] 
	 			\draw  [fill={rgb, 255:red, 0; green, 0; blue, 0 }  ,fill opacity=1 ] (58.66,175.22) .. controls (58.66,170.94) and (62.13,167.47) .. (66.41,167.47) .. controls (70.69,167.47) and (74.16,170.94) .. (74.16,175.22) .. controls (74.16,179.5) and (70.69,182.97) .. (66.41,182.97) .. controls (62.13,182.97) and (58.66,179.5) .. (58.66,175.22) -- cycle ;
	 			%Shape: Circle [id:dp6663069941603319] 
	 			\draw  [fill={rgb, 255:red, 0; green, 0; blue, 0 }  ,fill opacity=1 ] (344.91,175.22) .. controls (344.91,170.94) and (348.38,167.47) .. (352.66,167.47) .. controls (356.94,167.47) and (360.41,170.94) .. (360.41,175.22) .. controls (360.41,179.5) and (356.94,182.97) .. (352.66,182.97) .. controls (348.38,182.97) and (344.91,179.5) .. (344.91,175.22) -- cycle ;
	 			%Shape: Circle [id:dp015603782680160938] 
	 			\draw  [color={rgb, 255:red, 208; green, 2; blue, 27 }  ,draw opacity=1 ][dash pattern={on 3.75pt off 3pt on 7.5pt off 1.5pt}] (157.41,176.22) .. controls (157.41,171.94) and (160.88,168.47) .. (165.16,168.47) .. controls (169.44,168.47) and (172.91,171.94) .. (172.91,176.22) .. controls (172.91,180.5) and (169.44,183.97) .. (165.16,183.97) .. controls (160.88,183.97) and (157.41,180.5) .. (157.41,176.22) -- cycle ;
	 			%Shape: Circle [id:dp23039413915686469] 
	 			\draw  [color={rgb, 255:red, 208; green, 2; blue, 27 }  ,draw opacity=1 ][dash pattern={on 3.75pt off 3pt on 7.5pt off 1.5pt}] (224.91,175.22) .. controls (224.91,170.94) and (228.38,167.47) .. (232.66,167.47) .. controls (236.94,167.47) and (240.41,170.94) .. (240.41,175.22) .. controls (240.41,179.5) and (236.94,182.97) .. (232.66,182.97) .. controls (228.38,182.97) and (224.91,179.5) .. (224.91,175.22) -- cycle ;
	 			%Straight Lines [id:da8775176650443864] 
	 			\draw [color={rgb, 255:red, 245; green, 166; blue, 35 }  ,draw opacity=1 ] [dash pattern={on 4.5pt off 4.5pt}]  (68.41,74.47) -- (192.16,74.47) ;
	 			\draw [shift={(194.16,74.47)}, rotate = 180] [fill={rgb, 255:red, 245; green, 166; blue, 35 }  ,fill opacity=1 ][line width=0.08]  [draw opacity=0] (12,-3) -- (0,0) -- (12,3) -- cycle    ;
	 			\draw [shift={(66.41,74.47)}, rotate = 0] [fill={rgb, 255:red, 245; green, 166; blue, 35 }  ,fill opacity=1 ][line width=0.08]  [draw opacity=0] (12,-3) -- (0,0) -- (12,3) -- cycle    ;
	 			%Straight Lines [id:da3684632484181748] 
	 			\draw [color={rgb, 255:red, 245; green, 166; blue, 35 }  ,draw opacity=1 ] [dash pattern={on 4.5pt off 4.5pt}]  (226.91,167.47) -- (350.66,167.47) ;
	 			\draw [shift={(352.66,167.47)}, rotate = 180] [fill={rgb, 255:red, 245; green, 166; blue, 35 }  ,fill opacity=1 ][line width=0.08]  [draw opacity=0] (12,-3) -- (0,0) -- (12,3) -- cycle    ;
	 			\draw [shift={(224.91,167.47)}, rotate = 0] [fill={rgb, 255:red, 245; green, 166; blue, 35 }  ,fill opacity=1 ][line width=0.08]  [draw opacity=0] (12,-3) -- (0,0) -- (12,3) -- cycle    ;
	 			%Straight Lines [id:da6038981732148653] 
	 			\draw [color={rgb, 255:red, 74; green, 144; blue, 226 }  ,draw opacity=1 ]   (263,63) -- (296.41,63.21) ;
	 			\draw [shift={(298.41,63.22)}, rotate = 180.35] [color={rgb, 255:red, 74; green, 144; blue, 226 }  ,draw opacity=1 ][line width=0.75]    (10.93,-3.29) .. controls (6.95,-1.4) and (3.31,-0.3) .. (0,0) .. controls (3.31,0.3) and (6.95,1.4) .. (10.93,3.29)   ;
	 			%Straight Lines [id:da43996729693951764] 
	 			\draw [color={rgb, 255:red, 74; green, 144; blue, 226 }  ,draw opacity=1 ]   (226.59,67.79) -- (260,68) ;
	 			\draw [shift={(224.59,67.78)}, rotate = 0.35] [color={rgb, 255:red, 74; green, 144; blue, 226 }  ,draw opacity=1 ][line width=0.75]    (10.93,-3.29) .. controls (6.95,-1.4) and (3.31,-0.3) .. (0,0) .. controls (3.31,0.3) and (6.95,1.4) .. (10.93,3.29)   ;
	 			%Straight Lines [id:da25678860946401527] 
	 			\draw [color={rgb, 255:red, 74; green, 144; blue, 226 }  ,draw opacity=1 ]   (131.59,152.79) -- (165,153) ;
	 			\draw [shift={(129.59,152.78)}, rotate = 0.35] [color={rgb, 255:red, 74; green, 144; blue, 226 }  ,draw opacity=1 ][line width=0.75]    (10.93,-3.29) .. controls (6.95,-1.4) and (3.31,-0.3) .. (0,0) .. controls (3.31,0.3) and (6.95,1.4) .. (10.93,3.29)   ;
	 			%Straight Lines [id:da5987051680485992] 
	 			\draw [color={rgb, 255:red, 74; green, 144; blue, 226 }  ,draw opacity=1 ]   (159,161) -- (192.41,161.21) ;
	 			\draw [shift={(194.41,161.22)}, rotate = 180.35] [color={rgb, 255:red, 74; green, 144; blue, 226 }  ,draw opacity=1 ][line width=0.75]    (10.93,-3.29) .. controls (6.95,-1.4) and (3.31,-0.3) .. (0,0) .. controls (3.31,0.3) and (6.95,1.4) .. (10.93,3.29)   ;
	 			
	 			% Text Node
	 			\draw (47,40) node [anchor=north west][inner sep=0.75pt]   [align=left] {Bounce back and force $\displaystyle n_{1}$ times};
	 			% Text Node
	 			\draw (185,135) node [anchor=north west][inner sep=0.75pt]   [align=left] {Bounce back and force $\displaystyle n_{2}$ times};
	 			% Text Node
	 			\draw (24,96) node [anchor=north west][inner sep=0.75pt]   [align=left] {the left electron};
	 			% Text Node
	 			\draw (23,194) node [anchor=north west][inner sep=0.75pt]   [align=left] {the left electron};
	 			% Text Node
	 			\draw (310,97) node [anchor=north west][inner sep=0.75pt]   [align=left] {the right electron};
	 			% Text Node
	 			\draw (307,191) node [anchor=north west][inner sep=0.75pt]   [align=left] {the right electron};
	 			% Text Node
	 			\draw (190,96) node [anchor=north west][inner sep=0.75pt]   [align=left] {nucleus};
	 			% Text Node
	 			\draw (187,192) node [anchor=north west][inner sep=0.75pt]   [align=left] {nucleus};

	 		\end{tikzpicture}
	 		
	 	\end{figure}
 
	 Based on the results presented in this paper, one can pursue further in several possibilities. For example, we can investigate the stability of periodic orbits for ${e^{-}}Z^{2+}e^{-}$ Helium atom on the plane with mean interaction. We can also study similar periodic orbits for the $n$-electron atom, where multiple electrons positioned on the opposite sides of the nucleus. Moreover, we can attempt to construct the Morse homology and the equivariant Morse homology for the  ${e^{-}}Z^{2+}e^{-}$  mean Helium atom, and subsequently explore the spectrum of the mean ${e^{-}}Z^{2+}e^{-}$ Helium. We leave them for the future.

	\section{Barutello-Ortega-Verzini regularization}\label{section2}
	
	In \cite{barutello2021regularized}, Barutello, Ortega and Verzini present a non-local regularization based on the Levi-Civita map and the Kustaanheimo-Stiefel map. They use this method to study the perturbed Kepler problem. Here, we briefly describe this method for the Kepler problem and state some useful results from \cite{cieliebak2022variational}.
	
	First, we give some conventions. We abbreviate the circle $S^1$ by $\mathbb{R}/\mathbb{Z}$. For any $z_1,z_2\in L^2(S^1,\mathbb{R})$, $$\left\langle z_1,z_2\right\rangle:=\int_{0}^{1}z_1(\tau)z_2(\tau)d\tau$$ is the $L^2$-inner product, and the $L^2$-norm of $z\in L^2(S^1,\mathbb{R})$ is denoted by $$\left\| z \right\|:=\sqrt{\left\langle z,z \right\rangle }.$$
	
	Consider two maps $q: S^1\to \mathbb{R}_{\ge 0}$, $z: S^1\to \mathbb{R}$ related by the Levi-Civita transformation %$\Phi_{LC}:\mathbb{C}\to \mathbb{C}$, 
	\begin{equation}\label{Levi-Civita transformation1}
	z\mapsto q(t)=(z(\tau))^2,
	\end{equation}
     where the time change $\tau(t)$ satisfies $\tau(0)=0$ and
     \begin{equation}\label{time change1}
     \frac{dt}{q(t)}=\frac{d\tau}{{\left\| z \right\|}^2}.
     \end{equation} 
	
	Then, we have the following calculations: the mean value of $q$ is $$\bar{q}:=\int_{0}^{1}{q\left( t \right)dt=\int_{0}^{1}{\frac{{{\left( z\left( \tau  \right) \right)}^{4}}}{{{\left\| z \right\|}^{2}}}d\tau =\frac{{{\left\| {{z}^{2}} \right\|}^{2}}}{{{\left\| z \right\|}^{2}}}}},$$  the integral $\int_{0}^{1}{\frac{dt}{q\left( t \right)}}$ is $$\int_{0}^{1}{\frac{dt}{q\left( t \right)}=\int_{0}^{1}{\frac{d\tau }{{{\left\| z \right\|}^{2}}}=\frac{1}{{{\left\| z \right\|}^{2}}}}},$$ the first derivative of $q$ is $$\dot{q}\left( t \right)=2z\left( \tau  \right){z}'\left( \tau  \right)\frac{d\tau }{dt}=\frac{2{{\left\| z \right\|}^{2}}{z}'\left( \tau  \right)}{z\left( \tau  \right)},$$  the second derivative of $q$ is $$\ddot{q}\left( t \right)=\frac{2{{\left\| z \right\|}^{2}}\left( {z}''\left( \tau  \right)z\left( \tau  \right)-{{\left( {z}'\left( \tau  \right) \right)}^{2}} \right)}{{{\left( z\left( \tau  \right) \right)}^{2}}}\frac{d\tau }{dt}=\frac{2{{\left\| z \right\|}^{4}}\left( {z}''\left( \tau  \right)z\left( \tau  \right)-{{\left( {z}'\left( \tau  \right) \right)}^{2}} \right)}{{{\left( z\left( \tau  \right) \right)}^{4}}}=\frac{1}{q\left( t \right)}\left( \frac{2{{\left\| z \right\|}^{4}}{z}''\left( \tau  \right)}{z\left( \tau  \right)}-\frac{{{\left( \dot{q}\left( t \right) \right)}^{2}}}{2} \right),$$ and the $L^2$-norm of $\dot{q}$ is $${{\left\| {\dot{q}} \right\|}^{2}}=\int_{0}^{1}{{{\left( \dot{q}\left( t \right) \right)}^{2}}dt=\int_{0}^{1}{\frac{4{{\left\| z \right\|}^{4}}{{\left( {z}'\left( \tau  \right) \right)}^{2}}}{{{\left( z\left( \tau  \right) \right)}^{2}}}\frac{{{\left( z\left( \tau  \right) \right)}^{2}}}{{{\left\| z \right\|}^{2}}}d\tau =4}}{{\left\| z \right\|}^{2}}{{\left\| {{z}'} \right\|}^{2}}.$$
	
	If we are given $z\in C^0(S^1, \mathbb{R})$ be a continuous function, $Z_z:=z^{-1}(0)$ is its zero set and it is finite. We give a $C^1$ map $t_z:S^1\to S^1$,	\begin{equation}\label{t_z}
		t_z(\tau):= \frac{\int_{0}^{\tau}(z(s)^2)ds}{{\left\| z \right\|}^{2}}.
	\end{equation}  
	
	\begin{lemma}{(Cieliebak, Frauenfelder and Volkov \cite{cieliebak2022variational})}\label{lemma2.1}
		If $z\in C^0(S^1, \mathbb{R})$ has only finitely many zeros, then the map $t_z:S^1\to S^1$ defined by (\ref{t_z}) is a homeomorphism.
	\end{lemma}
	
	Denoting the continuous inverse of $t_z$ is $\tau_z:=t^{-1}_z:S^1\to S^1$, then $\tau_z$ is of class $C^1$ on $S^1\backslash t_z(Z_z)$. We define a continuous map $q:S^1\to \mathbb{R}_{\ge 0}$,  \begin{equation}\label{transformation}
		q(t):=(z(\tau_z(t)))^2.
	\end{equation}
	
	So the two maps $z$ and $q$ are related by the Levi-Civita transformation (\ref{Levi-Civita transformation1}) in the sense that $\tau=\tau_z$. Their zero sets $Z_z$ and $Z_q:=q^{-1}(0)=t_z(Z_z)$ are in 1-to-1 correspondence via $t_z$ or $\tau_z$.
	
	If we are given a map $q\in C^0(S^1, \mathbb{R}_{\ge 0})$ with finite zero set $Z_q$ and satisfying $\int_{0}^{1}{\frac{ds}{q\left( s \right)}<\infty }$. We define $\tau_q:S^1\to S^1$, \begin{equation}\label{tauq}
		\tau_{q}(t):=(\int_{0}^{1}\frac{ds}{q\left( s \right)})^{-1}\int_{0}^{t}\frac{ds}{q\left( s \right)}.
	\end{equation}
	
	Lemma 2.1 in \cite{barutello2021regularized} shows that the map $\tau_q:S^1\to S^1$ is a homeomorphism whose inverse is $t_q:={\tau_q}^{-1}$, and $t_q(1)={\tau}^{-1}_{q}(1)=1$. We can define a continuous function $z:S^1\to \mathbb{R}$ by \begin{equation}\label{inverse transformation}
		z(\tau)^2:=q(t_q(\tau)).
	\end{equation}
	
	So the zero set $Z_z=\tau_{q}(Z_q)$ is finite, and we associate to $z$ the homeomorphism $t_z$ and $\tau_z$, then $\tau_q=\tau_z$ and $t_q=t_z$. Then $q$ and $z$ are related by the Levi-Civita transformation (\ref{transformation}). And if $Z_z$ consists of even number of points, then we can determine $z$ in the sense of a global sign by the requirement that $z$ switches its sign at each zero. Then from \cite{cieliebak2022variational}, we have the following lemma:
	
	\begin{lemma}{(Cieliebak, Frauenfelder and Volkov \cite{cieliebak2022variational})}\label{lemma2.2}
		The Levi-Civita transformation (\ref{transformation}) defines, for each integer $m\in 2\mathbb{N}$, a surjective 2-to-1 map 
		\begin{equation*}
			\begin{aligned}
				 \Phi_{LC}:& \left\{ \left. z\in {{C}^{0}}\left( {{S}^{1}},\mathbb{R} \right) \right|z\text{ has precisely }m\text{ zeros and switches sign at each zero} \right\} \\ 
				& \to \left\{ \left. q\in {{C}^{0}}\left( {{S}^{1}},{{\mathbb{R}}_{\ge 0}} \right) \right|q\text{ has precisely }m\text{ zeros and }\int_{0}^{1}{\frac{ds}{q\left( s \right)}<\infty } \right\}. \\ 
			\end{aligned}
		\end{equation*}
	\end{lemma} 
	  
	  Near each zero $t_{*}$ of $q$, denoting $s_{*}=\left\{ \begin{aligned}
	  	& -1,\text{ }t< t_{*} \\ 
	  	& +1,\text{ }t>t_{*} \\ 
	  \end{aligned} \right.$ be the local sign function. The Newton's equation of the Kepler problem of a body of mass 1 moving in the gravitational field of a body of mass $N$ is $$\ddot{q}(t)=-\frac{N}{(q(t))^2}.$$
	  
		If $q\in C^1$ outside its zero set, we denote its Kepler energy at time $t$ by $E_q(t):=\frac{(\dot{q}(t))^2}{2}-\frac{N}{q(t)}(t\in S^1\backslash Z_q)$ and $N>0$ is fixed. Then according to the Levi-Civita transformation, we get $E_z(\tau):=\frac{2{\left\| {z} \right\|}^{4}(z'(\tau))^{2}-N}{(z(\tau))^2}$, $\tau\in S^1\backslash Z_z$.
	
	\begin{lemma}{(Cieliebak, Frauenfelder and Volkov \cite{cieliebak2022variational})}
	Let $z$, $q$ be as in Lemma \ref{lemma2.2} related by the Levi-Civita transformation (\ref{transformation}), and let $l$ be a non-negative integer. Then the following hold:
	
	(a) $z\in H^1(S^1, \mathbb{R})$ if and only if $q\in H^1(S^1, \mathbb{R}_{\ge 0})$;
	
	(b) $z$ is of class $C^l$ outside its zeros if and only if $q$ is of class $C^l$ outside its zeros;
	
	(c)  $z$ is of class $C^1$ outside $Z_z$ if and only if $q$ is of class $C^1$ outside $Z_q$ and for each $t_*\in Z_q$, $\underset{t_{*}\ne t\to t_{*}}{\mathop{\lim }}\,s_{*}(t){\sqrt{q(t)}\dot{q}(t)}$ exists;
	
	(d) $z$ is of class $C^1$ with transverse zeros if and only if $q$ is of class $C^1$ outside $Z_q$ and for each $t_*\in Z_q$ the limit in (c) exists and is positive;
	
	(e) the energy $E_z:S^1\backslash Z_z\to \mathbb{R}$ is defined and extends to a continuous function $S^1\to \mathbb{R}$ if and only if $E_q:S^1\backslash Z_q\to \mathbb{R}$ is defined and extends to a continuous function $S^1\to \mathbb{R}$;
	
	(f) the conditions in (e) imply those in (d).
\end{lemma}

	\begin{corollary}{(Cieliebak, Frauenfelder and Volkov \cite{cieliebak2022variational})}\label{Lemma2.3}
		For each even $m\in 2\mathbb{N}$, the Levi-Civita map $\Phi_{LC}$ of Lemma \ref{lemma2.2} restricts to a surjective 2-to-1 map  $$\Phi_{LC}:\mathcal{C}^{1}_{ce}(S^1,\mathbb{R})\to \mathscr{H}^{1}_{ce}(S^1,\mathbb{R}_{\ge 0}),$$ where
		
		\noindent $\bullet $ $\mathcal{C}^{1}_{ce}(S^1,\mathbb{R})$ denotes the set of $z\in C^1(S^1,\mathbb{R})$ with precisely $m$ zeros such that all zeros are transverse and the energy $E_z$ extends to a continuous function $S^1\to \mathbb{R}$ and 
		
		\noindent $\bullet $  $\mathscr{H}^{1}_{ce}(S^1,\mathbb{R}_{\ge 0})$ denotes the set of $q\in H^1(S^1,\mathbb{R}_{\ge 0})$ with precisely $m$ zeros such that $q$ is of class $C^1$ outside its zeros and the energy $E_q$ extends  to a continuous function  $S^1\to \mathbb{R}$.
	\end{corollary}
	
	\section{Variational approach to the $e^{-}Z^{2+}e^{-}$ Helium with mean interaction}
    The Newton's equation describing an electron moving in the electric field of a fixed nucleus with charge $N>0$ is $$\ddot{q}(t)=-\frac{N}{(q(t))^2},$$ which is the same formula as that of the Kepler problem shown before. Here, in the $e^{-}Z^{2+}e^{-}$ Helium, the charge $N=2$.
    
    Assume the nucleus lies in the origin point on the real line. Let the distance between one electron and the nucleus is $q_1$, the distance between the other electron and the nucleus is  $q_2$, $q_1>0,q_2>0$. In this case, we don't know the exact relationship of size between $q_1$ and $q_2$, and we denote $\bar{{{q}_{i}}}=\int_{0}^{1}{ {{q}_{i}}\left( t \right) dt}\left(i=0,1\right)$.
    
    In this paper, we study the mean interaction equations of the $e^{-}Z^{2+}e^{-}$ Helium: \begin{equation}\label{main equation of mean interaction}
    	\begin{aligned}
    		&\ddot{q_1}=-\frac{2}{\left(q_1\right)^2}+\frac{1}{ \left(\bar{q}_{1}+\bar{q}_{2}\right)^2},\\
    		&\ddot{q_2}=-\frac{2}{\left(q_2\right)^2}+\frac{1}{ \left(\bar{q}_{1}+\bar{q}_{2}\right)^2}.
    \end{aligned}
    \end{equation}
    \subsection{Variational characterization of generalized solutions}
    
    Solutions of (\ref{main equation of mean interaction}) avoiding the origin are the critical points of the action functional
    
     \begin{equation}
     	\tilde{\mathcal{S}}_{av}\left(q_1,q_2\right):=\sum_{i=1}^{2}\left(\frac{1}{2}\int_{0}^{1}({\dot{q}_i}(t))^{2}dt+\int_{0}^{1}\frac{2}{q_i(t)}dt\right)-\frac{1}{\bar{q}_{1}+\bar{q}_{2}},
     \end{equation}   
   where $\tilde{\mathcal{S}}_{av}:\tilde{\mathscr{H}}_{ur}^1\to \mathbb{R}$, $\tilde{\mathscr{H}}_{ur}^1:=\left\{q=(q_1,q_2)\in H^1(S^1,\mathbb{R}_{\ge 0}^2)\mid q_1>0, q_2>0\right\}$ is an open subset of Hilbert space $H^1(S^1,\mathbb{R}_{\ge 0}^2)$.
    
    In fact, this unregularized functional has no critical points because all periodic solutions have collisions. We use the regularization described in Section \ref{section2} to study the motions of two electrons in the $e^{-}Z^{2+}e^{-}$ mean Helium.
     
     For $i=1,2$, let $q_i$ and $z_i$ be related by Levi-Civita transformations \begin{equation}\label{ Levi-Civita transformations2}
     	q_i(t)=z_i{(\tau_i(t))}^2,
     \end{equation}
     for time changes $\tau_i(t)$ satisfying $\tau_i(0)=0$ and
     \begin{equation}\label{time change2}
     	\frac{dt}{q_i(t)}=\frac{d\tau_i(t)}{\parallel z_i\parallel ^2}.
     \end{equation} 
     
     So we can rewrite the action functional in terms of the $z_i$，$$-\frac{1}{\bar{q}_1+\bar{q}_2}=-\frac{1}{\frac{\parallel z_1^2\parallel ^2}{\parallel z_1\parallel ^2}+\frac{\parallel z_{2}^{2}\parallel ^2}{\parallel z_2\parallel ^2}}=-\frac{\parallel z_1\parallel ^2\parallel z_2\parallel ^2}{\parallel z_{1}^{2}\parallel ^2\parallel z_2\parallel ^2+\parallel z_{2}^{2}\parallel ^2\parallel z_1\parallel ^2},$$
     
      $$\frac{1}{2}\int_{0}^{1}({\dot{q}_i}(t))^{2}dt+\int_{0}^{1}\frac{2}{q_i(t)}dt=2\parallel z_i\parallel ^2\parallel z_{i}^{\prime}\parallel ^2+\frac{2}{\parallel z_i\parallel ^2}.$$
      
     We denote the resulting action functional with mean interaction of $\left(z_1,z_2\right)$ by
     \begin{equation}
     	\tilde{\mathscr{A}}\left(z_1,z_2\right):=-\frac{\parallel z_1\parallel ^2\parallel z_2\parallel ^2}{\parallel z_{1}^{2}\parallel ^2\parallel z_2\parallel ^2+\parallel z_{2}^{2}\parallel ^2\parallel z_1\parallel ^2},
     \end{equation} 
 
     \begin{equation}
     	\tilde{\mathscr{Q}}\left( z_i \right) =2\parallel z_i\parallel ^2\parallel z_{i}^{\prime}\parallel ^2+\frac{2}{\parallel z_i\parallel ^2}(i=1,2).
     \end{equation}
     
     These quantities are naturally defined on the space
     \begin{equation}
     	\tilde{\mathscr{H}}_{av}^{1}:=\left\{ z=\left( z_1,z_2 \right) \in H^1\left( S^1,\mathbb{R}^2 \right) \mid \parallel z_1\parallel >0,\parallel z_2\parallel >0 \right\},
     \end{equation} 
    where $\tilde{\mathscr{H}}_{av}^{1}$ is an open subset of the Hilbert space $H^1\left( S^1,\mathbb{R}^2 \right)$.
     
      On $\tilde{\mathscr{H}}_{av}^{1}$, we consider the functional $\tilde{\mathscr{B}}_{av}:\tilde{\mathscr{H}}_{av}^{1} \to \mathbb{R}$,
     \begin{equation}\label{Regularized action functional}
     	\tilde{\mathscr{B}}_{av}(z_1,z_2):=\tilde{\mathscr{Q}}(z_1)+\tilde{\mathscr{Q}}(z_2)+\tilde{\mathscr{A}}\left(z_1,z_2\right).
     \end{equation}  
     
     Due to the construction of (\ref{transformation}) and (\ref{inverse transformation}), there is a torus action acting as a time shift both on $\tilde{\mathscr{H}}_{ur}^1$ and $\tilde{\mathscr{H}}_{av}^{1}$, and this action is free. In other words, $\forall (\theta_{1},\theta_{2})\in \mathbb{T}^2$,  $(\theta_1,\theta_2)_{*}(q_1,q_2)=(q_1(t+\theta_1),q_2(t+\theta_2))$ and $(\theta_1,\theta_2)_{*}(z_1,z_2)=(z_1(\tau+\theta_1),z_2(\tau+\theta_2))$.

     \begin{lemma}
     	The torus action is an equivariant map from $\tilde{\mathscr{H}}_{av}^{1}$ to $\tilde{\mathscr{H}}_{ur}^1$, with respect to the Levi-Civita transformation (\ref{ Levi-Civita transformations2}) between $z_i$ and $q_i$, $i=1,2$. i.e. the diagram below is commutative, where $\Phi _{LC}\times \Phi _{LC} ( z) =(z_1( \tau _{z_1}( t))^{2},z_2( \tau _{z_2}( t))^{2}) =:(q_1(t),q_2(t))=q( t)$.
     	
     	\text{   }
     	
     	\tikzset{every picture/.style={line width=0.75pt}} %set default line width to 0.75pt        
     	\centering

     	\tikzset{every picture/.style={line width=0.75pt}} %set default line width to 0.75pt        
     	
     	\begin{tikzpicture}[x=0.75pt,y=0.75pt,yscale=-1,xscale=1]
     		%uncomment if require: \path (0,198); %set diagram left start at 0, and has height of 198
     		
     		%Straight Lines [id:da31713639922328096] 
     		\draw    (256.71,47.35) -- (388.69,47.7) ;
     		\draw [shift={(390.69,47.71)}, rotate = 180.15] [color={rgb, 255:red, 0; green, 0; blue, 0 }  ][line width=0.75]    (10.93,-3.29) .. controls (6.95,-1.4) and (3.31,-0.3) .. (0,0) .. controls (3.31,0.3) and (6.95,1.4) .. (10.93,3.29)   ;
     		%Straight Lines [id:da2970887814584948] 
     		\draw    (257.72,128.52) -- (389.7,128.87) ;
     		\draw [shift={(391.7,128.87)}, rotate = 180.15] [color={rgb, 255:red, 0; green, 0; blue, 0 }  ][line width=0.75]    (10.93,-3.29) .. controls (6.95,-1.4) and (3.31,-0.3) .. (0,0) .. controls (3.31,0.3) and (6.95,1.4) .. (10.93,3.29)   ;
     		%Straight Lines [id:da44086318184735007] 
     		\draw    (233.62,60.71) -- (234.98,113.52) ;
     		\draw [shift={(235.03,115.52)}, rotate = 268.52] [color={rgb, 255:red, 0; green, 0; blue, 0 }  ][line width=0.75]    (10.93,-3.29) .. controls (6.95,-1.4) and (3.31,-0.3) .. (0,0) .. controls (3.31,0.3) and (6.95,1.4) .. (10.93,3.29)   ;
     		%Straight Lines [id:da5625213883234771] 
     		\draw    (408.36,59.79) -- (409.72,112.6) ;
     		\draw [shift={(409.77,114.6)}, rotate = 268.52] [color={rgb, 255:red, 0; green, 0; blue, 0 }  ][line width=0.75]    (10.93,-3.29) .. controls (6.95,-1.4) and (3.31,-0.3) .. (0,0) .. controls (3.31,0.3) and (6.95,1.4) .. (10.93,3.29)   ;
     		%Shape: Arc [id:dp5834423796177967] 
     		\draw  [draw opacity=0] (294.43,71.71) .. controls (300.11,63.81) and (310.32,57.98) .. (321.8,57.05) .. controls (339,55.66) and (352.53,65.8) .. (352.01,79.7) .. controls (351.68,88.62) and (345.64,96.93) .. (336.81,102.09) -- (320.87,82.22) -- cycle ; \draw   (294.43,71.71) .. controls (300.11,63.81) and (310.32,57.98) .. (321.8,57.05) .. controls (339,55.66) and (352.53,65.8) .. (352.01,79.7) .. controls (351.68,88.62) and (345.64,96.93) .. (336.81,102.09) ;  
     		\draw   (350.06,100.33) -- (335.69,103.59) -- (338.47,88.81) ;
     		
     		% Text Node
     		\draw (215.62,36.71) node [anchor=north west][inner sep=0.75pt]    {$\tilde{\mathscr{H}}_{av}^{1}$};
     		% Text Node
     		\draw (218.63,117.88) node [anchor=north west][inner sep=0.75pt]    {$\tilde{\mathscr{H}}_{av}^{1}$};
     		% Text Node
     		\draw (396.36,39.78) node [anchor=north west][inner sep=0.75pt]    {$\tilde{\mathscr{H}}_{ur}^1$};
     		% Text Node
     		\draw (399.37,119.92) node [anchor=north west][inner sep=0.75pt]    {$\tilde{\mathscr{H}}_{ur}^1$};
     		% Text Node
     		\draw (275.01,24.16) node [anchor=north west][inner sep=0.75pt]    {$\Phi _{LC} \times \Phi _{LC}$};
     		% Text Node
     		\draw (173.48,73.7) node [anchor=north west][inner sep=0.75pt]    {$T_{( \theta _{1} ,\theta _{2})}{}$};
     		% Text Node
     		\draw (423.53,76.88) node [anchor=north west][inner sep=0.75pt]    {$T_{\left( \theta _{1}^{'} ,\theta _{2}^{'}\right)}$};
     		% Text Node
     		\draw (127.32,35.67) node [anchor=north west][inner sep=0.75pt]    {$z=( z_{1} ,z_{2}) \in $};
     		% Text Node
     		\draw (9.57,147.81) node [anchor=north west][inner sep=0.75pt]    {$\tilde{z} =T_{( \theta _{1} ,\theta _{2})} z=\left(\tilde{z_{1}} ,\tilde{z_{ 2}}\right) =( z_{1}( \tau +\theta _{1}) ,z_{2}( \tau +\theta _{2}))$};
     		% Text Node
     		\draw (222.23,148.28) node [anchor=north west][inner sep=0.75pt]  [rotate=-271.51]  {$\in $};
     		% Text Node
     		\draw (374.81,140.93) node [anchor=north west][inner sep=0.75pt]    {$\tilde{q} =T_{\left( \theta _{1}^{'} ,\theta _{2}^{'}\right)} q=\left(\tilde{q_{1}} ,\tilde{q_{ 2}}\right)=\left( q_{1}\left( t+\theta _{1}^{'}\right) ,q_{2}\left( t+\theta _{2}^{'}\right)\right)$};
     		% Text Node
     		\draw (399.98,146.28) node [anchor=north west][inner sep=0.75pt]  [rotate=-271.51]  {$\in $};
     		% Text Node
     		\draw (433.68,36.75) node [anchor=north west][inner sep=0.75pt]    {$q=( q_{1} ,q_{2})$};
     		% Text Node
     		\draw (433.56,53.9) node [anchor=north west][inner sep=0.75pt]  [rotate=-179.25]  {$\in $};
     		% Text Node
     		\draw (283.01,109.16) node [anchor=north west][inner sep=0.75pt]    {$\Phi _{LC} \times \Phi _{LC}$};

     	\end{tikzpicture}
     	
     \end{lemma}
     
     %补充证明
     \begin{proof}
     	From the definition of $t_z$ in (\ref{t_z}), $\forall \tau\in S^1$, $i=1,2$, we have $t_{T_{\theta_{i}}(z_i)}(\tau)=t_{z_i}(\tau+\theta_{i})-t_{z_i}(\theta_{i})$. Because $t_z:S^1\to S^1$ is a homeomorphism, $\tau_z=t_z^{-1}:S^1\to S^1$, so $t_{T_{\theta_{i}}(z_i)}(\tau_{T_{\theta_{i}}(z_i)}(t))=t$. Then we get $$t_{z_i}(\tau_{T_{\theta_{i}}(z_i)}(t)+\theta_{i})=t_{z_i}(\theta_{i})+t,$$ and furthermore $$\tau_{T_{\theta_{i}}(z_i)}(t)+\theta_{i}=t_{z_i}^{-1}(t+t_{z_i}(\theta_{i}))=\tau_{z_i}(t+t_{z_i}(\theta_{i})).$$
     	
     	Let $T_{\left( \theta _{1}^{'} ,\theta _{2}^{'}\right)}:=T_{(t_{z_1}(\theta_{1}),t_{z_2}(\theta_{2}))}$, then\begin{equation*}
     		\begin{aligned}
     			\tilde{q_i}&=(T_{\theta_{i}}{z_i}(\tau_{T_{\theta_{i}}{z_i}}(t))^2\\
     			&=(z(\tau_{T_{\theta_{i}}{z_i}}(t)+\theta_{i}))^2\\
     			&=(z(\tau_{z_i}(t+t_{z_i}(\theta_{i}))))^2\\
     			&=q(t+t_{z_i}(\theta_{i})).
     		\end{aligned}
     	\end{equation*} 
     	
     	Now we have $\Phi _{LC} \times \Phi _{LC}\circ T_{( \theta _{1} ,\theta _{2})}=T_{\left( \theta _{1}^{'} ,\theta _{2}^{'}\right)}\circ \Phi _{LC} \times \Phi _{LC}$, we prove this diagram above is commutative.
     \end{proof}
     
   So we can mod out this torus action. We denote $\mathscr{H}_{ur}^1:=\tilde{\mathscr{H}}_{ur}^1/{\mathbb{T}}^2$\label{H1av}, $\mathscr{H}_{av}^{1}:=\tilde{\mathscr{H}}_{av}^{1}/{\mathbb{T}}^2$.
   
   From the expression of $\tilde{\mathcal{S}}_{av}$ and $\tilde{\mathscr{B}}_{av}$, they are invariant under this torus action which comes from the time reparameterization.  We define $\mathcal{S}_{av}:\mathscr{H}_{ur}^1\to \mathbb{R}$,\begin{equation}\label{action functional of mean helium}
   	\mathcal{S}_{av}\left(q_1,q_2\right):=\sum_{i=1}^{2}\left(\frac{1}{2}\int_{0}^{1}({\dot{q}_i}(t))^{2}dt+\int_{0}^{1}\frac{2}{q_i(t)}dt\right)-\frac{1}{\bar{q}_{1}+\bar{q}_{2}},
   \end{equation}  
   and $\mathscr{B}_{av}:\mathscr{H}_{av}^{1} \to \mathbb{R}$,
   \begin{equation}\label{Bav}
   	\mathscr{B}_{av}(z_1,z_2):=\mathscr{Q}(z_1)+\mathscr{Q}(z_2)+\mathscr{A}\left(z_1,z_2\right),
   \end{equation}  with functionals \begin{equation}\label{Lagrangian}
   	\mathscr{Q}\left( z_i \right) =2\parallel z_i\parallel ^2\parallel z_{i}^{\prime}\parallel ^2+\frac{2}{\parallel z_i\parallel ^2}(i=1,2): \mathscr{H}_{av}^{1} \to \mathbb{R},
   \end{equation}
   \begin{equation}\label{mean interaction}
   	\mathscr{A}\left(z_1,z_2\right):=-\frac{\parallel z_1\parallel ^2\parallel z_2\parallel ^2}{\parallel z_{1}^{2}\parallel ^2\parallel z_2\parallel ^2+\parallel z_{2}^{2}\parallel ^2\parallel z_1\parallel ^2}:\mathscr{H}_{av}^{1} \to \mathbb{R}.
   \end{equation} 
      
      We call $(q_1,q_2)\in H^1\left( S^1,\mathbb{R}_{\ge 0}\times \mathbb{R}_{\ge 0} \right)$ a generalized solution of (\ref{main equation of mean interaction}) if for $i=1,2$ the following holds:
      
      (1)the zero sets $Z_i=q^{-1}_{i}(0)\subset S^1$ are finite and has an even number of elements;
      
      (2)the restrictions $q_i:S^1\setminus Z_i \to \mathbb{R}_{\ge 0}(i=1,2)$ are smooth and satisfy (\ref{main equation of mean interaction});
      
      (3)the energies $E_i(t):=\frac{{\dot{q}_i}(t)^2}{2}-\frac{2}{q_i(t)}-\frac{q_i(t)}{(\bar{q}_1+\bar{q}_2)^2}$, $t\in S^1\setminus Z_i$ extend to continuous functions $E_i:S^1\to \mathbb{R}$.
      
      \begin{remark}
      	In general, the definition of a generalized solution does not specifically distinguish the number of zeros of a function. Here, since both $q_i$ and $z_i$ are defined on the loop space $S^1$, $z_i$ must has an even number of zeros in order to represent periodic orbits. And according to Section \ref{section2}, we first consider generalized solutions with even number of zeros. 
      \end{remark}
      
      The exclusion of collisions is a significant challenge in the application of variational techniques, which motivates the study of generalized solutions derived from the variational approach to the study of selected trajectories in the $n$-body problem. In celestial mechanics, there is a long tradition of using the notion of generalized solutions. Many different definitions of generalized solutions have been given in the literature on variational methods, see e.g. [\cite{ambrosetti1990closed},\cite{bahri1991periodic},\cite{barutello2008singularities},\cite{coti1994collision},\cite{rabinowitz1994note},\cite{tanaka1993prescribed}], where \cite{barutello2008singularities} addressing non-autonomous systems. There are also several papers proving the existence of periodic generalized solutions to the one-dimensional perturbed Kepler problem, see e.g. [\cite{ortega2011linear},\cite{rebelo2018periodic},\cite{zhao2016some}]. 
      
%      \begin{thm}({Barutello, Ortega and Verzini\cite{barutello2021regularized}})\label{main reference theorem}
%      	Under the Levi-Civita transformation (\ref{ Levi-Civita transformations2}) with time change (\ref{time change2}), for $i=1,2$, critical points $z_i:S^1\to \mathbb{R}$ of the functionals $\mathscr{D}(z_i)$ defined in (\ref{Lagrangian}) are in 2-to-1 correspondence with generalized solutions with even number of zeros $q_i: S^1\to \mathbb{R}_{\ge 0}$ of (\ref{main equation of mean interaction}) separately.
%      \end{thm}    
%      
%      We will not go into the proof of this Theorem, the details can be found in \cite{cieliebak2022variational}. From this Theorem, we can get the following:
      
%      \begin{thm}{(Generalized solutions with mean interaction)}\label{4-1对应定理}
%      	Under the Levi-Civita transformations (\ref{ Levi-Civita transformations2}) with time change (\ref{time change2}), critical points $(z_1,z_2)$ of the action functional $\mathscr{B}_{av}$ are in 4-to-1 correspondence with generalized solutions $(q_1,q_2)$ of (\ref{main equation of mean interaction}).
%      \end{thm}
  
      We will prove Theorem \ref{Main1} in details in the following four subsections.
      
      \subsection{The differential of $\mathscr{B}_{av}$}\label{section3.2}
      The differential of the mean interaction $\mathscr{A}$ at $(z_1,z_2)\in \mathscr{H}_{av}^{1}$ in the direction $(v_1,v_2)\in H^1\left( S^1,\mathbb{R}^2 \right)$ is given by
      
      \begin{equation*}
      \begin{aligned}
      	 D\mathscr{A}\left[ {{z}_{1}},{{z}_{2}} \right]\left( {{v}_{1}},{{v}_{2}} \right) &=-2\frac{{{\left\| {{z}_{2}} \right\|}^{2}}\cdot \left\langle {{z}_{1}},{{v}_{1}} \right\rangle +{{\left\| {{z}_{1}} \right\|}^{2}}\cdot \left\langle {{z}_{2}},{{v}_{2}} \right\rangle }{{{\left\| z_{1}^{2} \right\|}^{2}}\cdot {{\left\| {{z}_{2}} \right\|}^{2}}+{{\left\| z_{2}^{2} \right\|}^{2}}\cdot {{\left\| {{z}_{1}} \right\|}^{2}}} \\ 
      	& \text{  }\text{  }\text{  }+2\frac{{{\left\| {{z}_{1}} \right\|}^{2}}\cdot {{\left\| {{z}_{2}} \right\|}^{2}}\left( 2{{\left\| {{z}_{2}} \right\|}^{2}}\cdot \left\langle z_{1}^{3},{{v}_{1}} \right\rangle +{{\left\| z_{1}^{2} \right\|}^{2}}\cdot \left\langle {{z}_{2}},{{v}_{2}} \right\rangle  \right)}{{{\left( {{\left\| z_{1}^{2} \right\|}^{2}}\cdot {{\left\| {{z}_{2}} \right\|}^{2}}+{{\left\| z_{2}^{2} \right\|}^{2}}\cdot {{\left\| {{z}_{1}} \right\|}^{2}} \right)}^{2}}} \\ 
      	& \text{  }\text{  }\text{  }+2\frac{{{\left\| {{z}_{1}} \right\|}^{2}}\cdot {{\left\| {{z}_{2}} \right\|}^{2}}\left( 2{{\left\| {{z}_{1}} \right\|}^{2}}\cdot \left\langle z_{2}^{3},{{v}_{2}} \right\rangle +{{\left\| z_{2}^{2} \right\|}^{2}}\cdot \left\langle {{z}_{1}},{{v}_{1}} \right\rangle  \right)}{{{\left( {{\left\| z_{1}^{2} \right\|}^{2}}\cdot {{\left\| {{z}_{2}} \right\|}^{2}}+{{\left\| z_{2}^{2} \right\|}^{2}}\cdot {{\left\| {{z}_{1}} \right\|}^{2}} \right)}^{2}}} \\ 
      	& \text{                             =}-2\frac{{{\left\| {{z}_{2}} \right\|}^{4}}\cdot {{\left\| z_{1}^{2} \right\|}^{2}}}{{{\left( {{\left\| z_{1}^{2} \right\|}^{2}}\cdot {{\left\| {{z}_{2}} \right\|}^{2}}+{{\left\| z_{2}^{2} \right\|}^{2}}\cdot {{\left\| {{z}_{1}} \right\|}^{2}} \right)}^{2}}}\left\langle {{z}_{1}},{{v}_{1}} \right\rangle  \\ 
      	& \text{  }\text{  }\text{ }-2\frac{{{\left\| {{z}_{1}} \right\|}^{4}}\cdot {{\left\| z_{2}^{2} \right\|}^{2}}}{{{\left( {{\left\| z_{1}^{2} \right\|}^{2}}\cdot {{\left\| {{z}_{2}} \right\|}^{2}}+{{\left\| z_{2}^{2} \right\|}^{2}}\cdot {{\left\| {{z}_{1}} \right\|}^{2}} \right)}^{2}}}\left\langle {{z}_{2}},{{v}_{2}} \right\rangle  \\ 
      	& \text{  }\text{  }\text{ }+4\frac{{{\left\| {{z}_{1}} \right\|}^{2}}\cdot {{\left\| {{z}_{2}} \right\|}^{4}}}{{{\left( {{\left\| z_{1}^{2} \right\|}^{2}}\cdot {{\left\| {{z}_{2}} \right\|}^{2}}+{{\left\| z_{2}^{2} \right\|}^{2}}\cdot {{\left\| {{z}_{1}} \right\|}^{2}} \right)}^{2}}}\left\langle z_{1}^{3},{{v}_{1}} \right\rangle  \\ 
      	&\text{  }\text{  } \text{ }+4\frac{{{\left\| {{z}_{1}} \right\|}^{4}}\cdot {{\left\| {{z}_{2}} \right\|}^{2}}}{{{\left( {{\left\| z_{1}^{2} \right\|}^{2}}\cdot {{\left\| {{z}_{2}} \right\|}^{2}}+{{\left\| z_{2}^{2} \right\|}^{2}}\cdot {{\left\| {{z}_{1}} \right\|}^{2}} \right)}^{2}}}\left\langle z_{2}^{3},{{v}_{2}} \right\rangle.   
      \end{aligned}
    \end{equation*} 
    
    The differential of $\mathscr{Q}_i$ at $z_i\in H^1\left( S^1,\mathbb{R}^2 \right)\setminus \left\{0\right\}$ in the direction $v_i\in H^1\left( S^1,\mathbb{R}^2 \right)$ is $$D\mathscr{Q}_i(z_i)v_i=4{{\left\| {{z}'_{i}} \right\|}^{2}}\left\langle {{z}_{i}},{{v}_{i}} \right\rangle -4{{\left\| {{z}_{i}} \right\|}^{2}}\left\langle {{z}''_{i}},{{v}_{i}} \right\rangle -\frac{4}{{{\left\| {{z}_{i}} \right\|}^{4}}}\left\langle {{z}_{i}},{{v}_{i}} \right\rangle. $$
    
    So the differential of $\mathscr{B}_{av}$ is 
    \begin{equation}\label{the differntial of B_av }
    \begin{aligned}
    	D\mathscr{B}_{av} \left[ {{z}_{1}},{{z}_{2}} \right]\left( {{v}_{1}},{{v}_{2}} \right)&=4\sum\limits_{i=1}^{2}{\left( {{\left\| {{z}'_{i}} \right\|}^{2}}\left\langle {{z}_{i}},{{v}_{i}} \right\rangle -{{\left\| {{z}_{i}} \right\|}^{2}}\left\langle {{z}''_{i}},{{v}_{i}} \right\rangle -\frac{\left\langle {{z}_{i}},{{v}_{i}} \right\rangle}{{{\left\| {{z}_{i}} \right\|}^{4}}} \right)} \\
    	&\text{  }\text{  } \text{ }-2\frac{{{\left\| {{z}_{2}} \right\|}^{4}}\cdot {{\left\| z_{1}^{2} \right\|}^{2}}}{{{\left( {{\left\| z_{1}^{2} \right\|}^{2}}\cdot {{\left\| {{z}_{2}} \right\|}^{2}}+{{\left\| z_{2}^{2} \right\|}^{2}}\cdot {{\left\| {{z}_{1}} \right\|}^{2}} \right)}^{2}}}\left\langle {{z}_{1}},{{v}_{1}} \right\rangle   \\ 
    	& \text{  }\text{  }\text{ }-2\frac{{{\left\| {{z}_{1}} \right\|}^{4}}\cdot {{\left\| z_{2}^{2} \right\|}^{2}}}{{{\left( {{\left\| z_{1}^{2} \right\|}^{2}}\cdot {{\left\| {{z}_{2}} \right\|}^{2}}+{{\left\| z_{2}^{2} \right\|}^{2}}\cdot {{\left\| {{z}_{1}} \right\|}^{2}} \right)}^{2}}}\left\langle {{z}_{2}},{{v}_{2}} \right\rangle  \\ 
    	& \text{  }\text{  }\text{ }+4\frac{{{\left\| {{z}_{1}} \right\|}^{2}}\cdot {{\left\| {{z}_{2}} \right\|}^{4}}}{{{\left( {{\left\| z_{1}^{2} \right\|}^{2}}\cdot {{\left\| {{z}_{2}} \right\|}^{2}}+{{\left\| z_{2}^{2} \right\|}^{2}}\cdot {{\left\| {{z}_{1}} \right\|}^{2}} \right)}^{2}}}\left\langle z_{1}^{3},{{v}_{1}} \right\rangle \\ 
    	&\text{  }\text{  } \text{ }+4\frac{{{\left\| {{z}_{1}} \right\|}^{4}}\cdot {{\left\| {{z}_{2}} \right\|}^{2}}}{{{\left( {{\left\| z_{1}^{2} \right\|}^{2}}\cdot {{\left\| {{z}_{2}} \right\|}^{2}}+{{\left\| z_{2}^{2} \right\|}^{2}}\cdot {{\left\| {{z}_{1}} \right\|}^{2}} \right)}^{2}}}\left\langle z_{2}^{3},{{v}_{2}} \right\rangle.
   \end{aligned}
    \end{equation}
    
    \subsection{Critical points of $\mathscr{B}_{av}$}
    
    Equation (\ref{the differntial of B_av }) leads to the characterization of critical points of $\mathscr{B}_{av}$, denoting the set of all critical points of $\mathscr{B}_{av}$ by $\mathscr{C}_{\mathscr{B}_{av}}$.
     
    \begin{proposition}\label{proposition}
      A point $\left(z_1,z_2\right)\in \mathscr{H}_{av}^{1}$ is a critical point of  $\mathscr{B}_{av}$, i.e.  a point $\left(z_1,z_2\right)\in \mathscr{C}_{\mathscr{B}_{av}}$ if and only if $(z_1,z_2)$ is smooth and solves the system of (uncoupled) ODEs
      
      \begin{equation}\label{uncoupled ODE}
      	 \begin{aligned}
      		& {{z}''_{1}}={{a}_{1}}{{z}_{1}}+{{b}_{1}}z_{1}^{3}\\ 
      		& {{z}''_{2}}={{a}_{2}}{{z}_{2}}+{{b}_{2}}z_{2}^{3}\\ 
      	\end{aligned} 
      \end{equation} 
  with the constants $$a_1=\frac{{{\left\| {{z}'_{1}} \right\|}^{2}}}{{{\left\| {{z}_{1}} \right\|}^{2}}}-\frac{1}{{{\left\| {{z}_{1}} \right\|}^{6}}}-\frac{{{\left\| {{z}_{2}} \right\|}^{4}}\cdot {{\left\| z_{1}^{2} \right\|}^{2}}}{2{{\left\| {{z}_{1}} \right\|}^{2}}\cdot{{\left( {{\left\| z_{1}^{2} \right\|}^{2}}\cdot {{\left\| {{z}_{2}} \right\|}^{2}}+{{\left\| z_{2}^{2} \right\|}^{2}}\cdot {{\left\| {{z}_{1}} \right\|}^{2}} \right)}^{2}} },$$  $$a_2=\frac{{{\left\| {{z}'_{2}} \right\|}^{2}}}{{{\left\| {{z}_{2}} \right\|}^{2}}}-\frac{1}{{{\left\| {{z}_{2}} \right\|}^{6}}}-\frac{{{\left\| {{z}_{1}} \right\|}^{4}}\cdot {{\left\| z_{2}^{2} \right\|}^{2}}}{2{{\left\| {{z}_{2}} \right\|}^{2}}\cdot{{\left( {{\left\| z_{1}^{2} \right\|}^{2}}\cdot {{\left\| {{z}_{2}} \right\|}^{2}}+{{\left\| z_{2}^{2} \right\|}^{2}}\cdot {{\left\| {{z}_{1}} \right\|}^{2}} \right)}^{2}} },$$
  $$b_1=\frac{{{\left\| {{z}_{2}} \right\|}^{4}}}{{{\left( {{\left\| z_{1}^{2} \right\|}^{2}}\cdot {{\left\| {{z}_{2}} \right\|}^{2}}+{{\left\| z_{2}^{2} \right\|}^{2}}\cdot {{\left\| {{z}_{1}} \right\|}^{2}} \right)}^{2}}},$$ $$b_2=\frac{{{\left\| {{z}_{1}} \right\|}^{4}}}{{{\left( {{\left\| z_{1}^{2} \right\|}^{2}}\cdot {{\left\| {{z}_{2}} \right\|}^{2}}+{{\left\| z_{2}^{2} \right\|}^{2}}\cdot {{\left\| {{z}_{1}} \right\|}^{2}} \right)}^{2}}},$$ where $a_i<0, b_i>0, i=1,2$.
    \end{proposition}
\begin{proof}
	From the definition, the solutions $(z_1,z_2)$ of the equation(\ref{uncoupled ODE}) are periodic orbits, so the signs of all constants are naturally. $(z_1,z_2)$ is a critical point of $\mathscr{B}_{av}$ if and only if $z_1$ and $z_2$ both have weak second derivatives and satisfy equation(\ref{uncoupled ODE}). Bootstrapping these equations, we can get $z_1$ and $z_2$ are smooth, so this proposition follows.
\end{proof}

\begin{corollary}(Cieliebak, Frauenfelder and Volkov\cite{cieliebak2022variational})\label{corollary}
	Suppose that $\left(z_1,z_2\right)\in \mathscr{C}_{\mathscr{B}_{av}}$, then $z_1$ and $z_2$ have transverse zeros. In particular, their zero sets $${{Z}_{i}}=\left\{ \left. \tau \in {{S}^{1}} \right|{{z}_{i}}\left( \tau  \right)=0 \right\}(i=1,2)$$ are finite. 
\end{corollary}

\subsection{From critical points to generalized solutions}
 Let$\left(z_1,z_2\right)\in \mathscr{C}_{\mathscr{B}_{av}}$, Proposition \ref{proposition} tells us $z_1$ and $z_2$ are smooth and satisfy equation (\ref{uncoupled ODE}).
 
 We define $t_{z_i}:S^1\to S^1(i=1,2)$ as (\ref{t_z}) separately. Corollary \ref{corollary} shows the zero set of the map $z_i$ is finite, then from Lemma \ref{lemma2.1}, $t_{z_i}$ is a homeomorphism with continuous inverse $\tau_{z_i}:S^1\to S^1$. Then we can get $q_i(t):=z_i(\tau_{{{z}_{i}}})^2$ as defined in (\ref{transformation}).
 
 Then we calculate the first derivative of $q_i$,
 
 $${{\dot{q}}_{i}}\left( t \right)=2{{z}_{i}}\left( \tau_{z_i}\left( t \right) \right)\cdot {{z}'_{i}}\left( \tau_{z_i} \left( t \right) \right)\cdot {{\tau }'_{{{z}_{i}}}}\left( t \right)=\frac{2{{z}'_{i}}\left( \tau_{z_i} \left( t \right) \right)}{\left( \tau_{z_i} \left( t \right) \right){{\left\| {{z}_{i}}\left( \tau_{z_i} \left( t \right) \right) \right\|}^{2}}}{{\bar{z}}_{i}}\left( \tau_{z_i} \left( t \right) \right).$$
 
 Next we calculate the second derivative of $q_i(i=1,2)$ separately, it follows:
 \begin{equation*}
 \begin{split}
 	{{{\ddot{q}}}_{1}}&=\frac{1}{{{q}_{1}}\left( t \right)}\left( 2{{\left\| {{z}_{1}} \right\|}^{4}}\frac{{z''_1}\left( \tau  \right)}{{{z}_{1}}\left( \tau  \right)}-\frac{{{\left( \dot{q}_{1}\left( t \right) \right)}^{2}}}{2} \right) \\ 
 	& =\left( 2{{\left\| {z'_1} \right\|}^{2}}\cdot {{\left\| {{z}_{1}} \right\|}^{2}}-\frac{2}{{{\left\| {{z}_{1}} \right\|}^{2}}}-\frac{{{\left\| {{z}_{1}} \right\|}^{2}}{{\left\| {{z}_{2}} \right\|}^{4}}{{\left\| z_{1}^{2} \right\|}^{2}}}{{{\left( {{\left\| z_{1}^{2} \right\|}^{2}}\cdot {{\left\| {{z}_{2}} \right\|}^{2}}+{{\left\| z_{2}^{2} \right\|}^{2}}\cdot {{\left\| {{z}_{1}} \right\|}^{2}} \right)}^{2}}} \right)\frac{1}{{{q}_{1}}} \\ 
 	& \text{  }\text{  }\text{  }+\frac{2{{\left\| {{z}_{1}} \right\|}^{4}}{{\left\| {{z}_{2}} \right\|}^{4}}}{{{\left( {{\left\| z_{1}^{2} \right\|}^{2}}\cdot {{\left\| {{z}_{2}} \right\|}^{2}}+{{\left\| z_{2}^{2} \right\|}^{2}}\cdot {{\left\| {{z}_{1}} \right\|}^{2}} \right)}^{2}}}-\frac{\dot{q}_{1}^{2}}{2{{q}_{1}}} \\ 
 	& =\left( 2{{\left\| {z'_1} \right\|}^{2}}\cdot {{\left\| {{z}_{1}} \right\|}^{2}}-\frac{2}{{{\left\| {{z}_{1}} \right\|}^{2}}}-\frac{\frac{{{\left\| z_{1}^{2} \right\|}^{2}}}{{{\left\| {{z}_{1}} \right\|}^{2}}}}{{{\left( \frac{{{\left\| z_{1}^{2} \right\|}^{2}}}{{{\left\| {{z}_{1}} \right\|}^{2}}}+\frac{{{\left\| z_{2}^{2} \right\|}^{2}}}{{{\left\| {{z}_{2}} \right\|}^{2}}} \right)}^{2}}}-\frac{\dot{q}_{1}^{2}}{2} \right)\frac{1}{{{q}_{1}}}+\frac{2}{{{\left( \frac{{{\left\| z_{1}^{2} \right\|}^{2}}}{{{\left\| {{z}_{1}} \right\|}^{2}}}+\frac{{{\left\| z_{2}^{2} \right\|}^{2}}}{{{\left\| {{z}_{2}} \right\|}^{2}}} \right)}^{2}}} \\ 
 	& =\left( \frac{{{\left\| {{{\dot{q}}}_{1}} \right\|}^{2}}}{2}-\int_{0}^{1}{\frac{2}{{{q}_{1}}\left( s \right)}ds-\frac{{{{\bar{q}}}_{1}}}{{{\left( {{{\bar{q}}}_{1}}+{{{\bar{q}}}_{2}} \right)}^{2}}}-\frac{\dot{q}_{1}^{2}}{2}} \right)\frac{1}{{{q}_{1}}}+\frac{2}{{{\left( {{{\bar{q}}}_{1}}+{{{\bar{q}}}_{2}} \right)}^{2}}}, \\ 
 \end{split}
\end{equation*}
 $${{{\ddot{q}}}_{2}}=\left( \frac{{{\left\| {{{\dot{q}}}_{2}} \right\|}^{2}}}{2}-\int_{0}^{1}{\frac{2}{{{q}_{2}}\left( s \right)}ds-\frac{{{{\bar{q}}}_{2}}}{{{\left( {{{\bar{q}}}_{1}}+{{{\bar{q}}}_{2}} \right)}^{2}}}-\frac{\dot{q}_{2}^{2}}{2}} \right)\frac{1}{{{q}_{2}}}+\frac{2}{{{\left( {{{\bar{q}}}_{1}}+{{{\bar{q}}}_{2}} \right)}^{2}}}.$$
 
 Thus $q_i$ satisfies the ODE \begin{equation}\label{the equation of q_i}
 	{{{\ddot{q}}}_{i}}=(c_i-\frac{\dot{q}_{i}^{2}}{2})\frac{1}{{{q}_{i}}}+\frac{2}{{{\left( {{{\bar{q}}}_{1}}+{{{\bar{q}}}_{2}} \right)}^{2}}}
 \end{equation} with the constant 
 	\begin{equation}\label{the equation of c_i}
 		c_i=\frac{{{\left\| {{{\dot{q}}}_{i}} \right\|}^{2}}}{2}-\int_{0}^{1}\frac{2}{{{q}_{i}}\left( s \right)}ds-\frac{{{{\bar{q}}}_{i}}}{{{\left( {{{\bar{q}}}_{1}}+{{{\bar{q}}}_{2}} \right)}^{2}}}.
 	\end{equation}

 At the global maximum $t_{i}^{max}$ of $q_i$, the equation of $q_i$ becomes $\frac{c_i}{{{q}_{i}(t_{i}^{max})}}+\frac{2}{{{\left( {{{\bar{q}}}_{1}}+{{{\bar{q}}}_{2}} \right)}^{2}}}={{{\ddot{q}}}_{i}}(t_{i}^{max})\le 0,$ so $c_i\le -\frac{2q_i(t_{i}^{max})}{{{\left( {{{\bar{q}}}_{1}}+{{{\bar{q}}}_{2}} \right)}^{2}}}. $
 
 Now consider the smooth map ${\beta}_{i}:=\frac{{{{\ddot{q}}}_{i}}-\frac{1}{{{\left( {{{\bar{q}}}_{1}}+{{{\bar{q}}}_{2}} \right)}^{2}}}}{{{q}_{i}}}:\left( {{t}_{i}}^{-},{{t}_{i}}^{+} \right)\to \mathbb{R},$ where ${{t}_{i}}^{-}<{{t}_{i}}^{+}$ are two adjacent zeros of $q_i$. Then we can obtain \begin{equation}\label{the equation of beta_i}
 	{{\beta }_{i}}q_{i}^{2}={{c}_{i}}-\frac{\dot{q}_{i}^{2}}{2}+\frac{{{q}_{i}}}{{{\left( {{{\bar{q}}}_{1}}+{{{\bar{q}}}_{2}} \right)}^{2}}}.
 \end{equation} Because $q_i
 \le q_i(t_{i}^{max})$, thus we have ${{\beta }_{i}}q_{i}^{2}\le-\frac{\dot{q}_{i}^{2}}{2}-\frac{q_i(t_{i}^{max})}{{{\left( {{{\bar{q}}}_{1}}+{{{\bar{q}}}_{2}} \right)}^{2}}}<0$, then we know on $\left( {{t}_{i}}^{-},{{t}_{i}}^{+} \right)$, ${\beta_{i}}<0$.
 
 Differentiating (\ref{the equation of beta_i}), it comes to ${{\dot{\beta }}_{i}}q_{i}^{2}+2{{\beta }_{i}}{{q}_{i}}{{\dot{q}}_{i}}=-{{\ddot{q}}_{i}}{{\dot{q}}_{i}}+\frac{{{{\dot{q}}}_{i}}}{{{\left( {{{\bar{q}}}_{1}}+{{{\bar{q}}}_{2}} \right)}^{2}}}=-{{\beta }_{i}}{{q}_{i}}{{\dot{q}}_{i}}$ and due to $q_i>0$, we have ${{\dot{\beta }}_{i}}q_{i}=-3{{\beta }_{i}}{\dot{q}_{i}}.$ By integrating both sides of this equality, we conclude $\beta_{i}=-\frac{\mu_i}{q_{i}^{3}}$ on  $\left( {{t}_{i}}^{-},{{t}_{i}}^{+} \right)$ for some constant $\mu_i>0$. Therefore for $t\in\left( {{t}_{i}}^{-},{{t}_{i}}^{+} \right)$, we have ${{{\ddot{q}}}_{i}}=-\frac{{\mu}_{i}}{{q_i(t)}^{2}}+\frac{1}{{{\left( {{{\bar{q}}}_{1}}+{{{\bar{q}}}_{2}} \right)}^{2}}}.$
 
 Then for $t\in\left( {{t}_{i}}^{-},{{t}_{i}}^{+} \right)$, we can get ${{\mu }_{i}}=-\left( \frac{{{\left\| {{{\dot{q}}}_{i}} \right\|}^{2}}}{2}-\int_{0}^{1}{\frac{2}{{{q}_{i}}\left( s \right)}ds-\frac{{{{\bar{q}}}_{i}}}{{{\left( {{{\bar{q}}}_{1}}+{{{\bar{q}}}_{2}} \right)}^{2}}}-\frac{{{{\dot{q}}}_{i}}{{\left( t \right)}^{2}}}{2}} \right){{q}_{i}}\left( t \right)-\frac{{{q}_{i}}{{\left( t \right)}^{2}}}{{{\left( {{{\bar{q}}}_{1}}+{{{\bar{q}}}_{2}} \right)}^{2}}}.$ In particular, ${{\mu }_{i}}\text{=}\underset{{{t}_{i}}\to t_{i}^{\pm }}{\mathop{\lim }}\,\frac{{{{\dot{q}}}_{i}}{{\left( t \right)}^{2}}{{q}_{i}}\left( t \right)}{2}=2{{\left\| {{z}_{i}} \right\|}^{4}}{{z'_i}}{{\left( {{\tau }_{{{z}_{i}}}}\left( t_{i}^{\pm } \right) \right)}^{2}}.$
 
Then we get $\frac{{\mu}_{i}}{q_i(t)}=- \frac{{{\left\| {{{\dot{q}}}_{i}} \right\|}^{2}}}{2}+\int_{0}^{1}{\frac{2}{{{q}_{i}}\left( s \right)}ds+\frac{{{{\bar{q}}}_{i}}}{{{\left( {{{\bar{q}}}_{1}}+{{{\bar{q}}}_{2}} \right)}^{2}}}+\frac{{{{\dot{q}}}_{i}}{{\left( t \right)}^{2}}}{2}} -\frac{{{q}_{i}}{{\left( t \right)}}}{{{\left( {{{\bar{q}}}_{1}}+{{{\bar{q}}}_{2}} \right)}^{2}}},$ and integrating this yields $$\mu_i \int_{0}^{1}{\frac{1}{{{q}_{i}}\left( t \right)}dt=2\int_{0}^{1}{\frac{1}{{{q}_{i}}\left( s \right)}ds}},$$ so we get $\mu_i=2$.
 
Since $\mu_i=2$, for $i=1,2$, \begin{equation*}
	\begin{aligned}
		E_i(t)&=\frac{{{\dot{q}_i}(t)}^2}{2}-\frac{2}{q_i(t)}-\frac{q_i(t)}{(\bar{q}_1+\bar{q}_2)^2}\\
		&=\frac{{{\left\| {{{\dot{q}}}_{i}} \right\|}^{2}}}{2}-\int_{0}^{1}\frac{2}{q_i(s)}ds-\frac{\bar{{{q}_{i}}}}{(\bar{q}_1+\bar{q}_2)^2},\\
	\end{aligned}
\end{equation*}
 the right-hand side are continuous functions of $t\in\left[0,1\right]$, then we can see the continuity of $E_i$.
 
Therefore we recover the original equation (\ref{main equation of mean interaction}), and it shows that the critical points of $\mathscr{B}_{av}$ that we get are generalized solutions of the original equation  (\ref{main equation of mean interaction}).

\subsection{From generalized solutions to critical points}
Let $(q_1,q_2)\in H^1\left( S^1,\mathbb{R}_{\ge 0}^2\right)$ be a generalized solution of (\ref{main equation of mean interaction}). Integrating the constant energy of $q_i$, then $E_i=\int_{0}^{1}\frac{\dot{q}_{i}(t)^2}{2}dt-\int_{0}^{1}\frac{2}{q_{i}(t)}dt-\frac{\bar{{{q}_{i}}}}{(\bar{q}_1+\bar{q}_2)^2}$. Because $q_i\in H^1$, the first term is finite, and the last term is also finite from its definition, then it follows that $\int_{0}^{1}\frac{1}{q_{i}(t)}dt<\infty$.

We associate $q_i$ the time reparameterization as in (\ref{tauq}) and define $z_i(\tau)^2:=q_i(t_{q_i}(\tau))$ as in (\ref{inverse transformation}). Lemma \ref{Lemma2.3} implies the set ${\Phi_{LC}^{-1}}\left( {{q}_{1}} \right)\times {\Phi_{LC}^{-1}}\left( {{q}_{2}} \right)$ consists of four elements. Let $Z_{z_i}:={z_i}^{-1}(0)(i=1,2)$.

\begin{lemma}
	Any $(z_1,z_2)\in {\Phi_{LC}^{-1}}\left( {{q}_{1}} \right)\times {\Phi_{LC}^{-1}}\left( {{q}_{2}} \right)$ satisfies the critical point equation (\ref{uncoupled ODE}) on the complement of the set $Z_{z_1}\cup Z_{z_2}$.\end{lemma}
\begin{proof}
	The equation of $q_i$ is ${{{\ddot{q}}}_{i}}=-\frac{2}{{q_i(t)}^{2}}+\frac{1}{{{\left( {{{\bar{q}}}_{1}}+{{{\bar{q}}}_{2}} \right)}^{2}}}.$ 
	
	Set ${\beta}_{i}:=\frac{{{{\ddot{q}}}_{i}}-\frac{1}{{{\left( {{{\bar{q}}}_{1}}+{{{\bar{q}}}_{2}} \right)}^{2}}}}{{{q}_{i}}}$ on ${{S}^{1}}\backslash {{Z}_{{{q}_{i}}}}$, then we have $\beta_{i}=-\frac{2}{q_i^3}$. The derivative of $\beta_{i}$ with respect to time is ${{\dot{\beta }}_{i}}=-\frac{3{{\beta }_{i}}{{{\dot{q}}}_{i}}}{{{q}_{i}}}$, we multiply both sides by $q_i^2$, then we get ${{\dot{\beta }}_{i}}q_{i}^{2}=-3{{\beta }_{i}}{{\dot{q}}_{i}}{{q}_{i}}$ which equals to ${{\dot{\beta }}_{i}}q_{i}^{2}+2{{\beta }_{i}}{{\dot{q}}_{i}}{{q}_{i}}=-{{\beta }_{i}}{{\dot{q}}_{i}}{{q}_{i}}.$
	
	By the definition of $\beta_{i}$, it becomes ${{\dot{\beta }}_{i}}q_{i}^{2}+2{{\beta }_{i}}{{\dot{q}}_{i}}{{q}_{i}}=-{{\dot{q}}_{i}}{{\ddot{q}}_{i}}+\frac{{{{\dot{q}}}_{i}}}{{{\left( {{{\bar{q}}}_{1}}+{{{\bar{q}}}_{2}} \right)}^{2}}}$, integrating both sides of this from 0 to $t$ we have ${\beta _{i}}q_{i}^{2}=C_i-\frac{\dot{q}_{i}^{2}}{2}+\frac{{{q}_{i}}}{{{\left( {{{\bar{q}}}_{1}}\text{+}{{{\bar{q}}}_{2}} \right)}^{2}}}$, i.e. ${{\ddot{q}}_{i}}{{q}_{i}}=C_i-\frac{\dot{q}_{i}^{2}}{2}+\frac{2{{q}_{i}}}{{{\left( {{{\bar{q}}}_{1}}\text{+}{{{\bar{q}}}_{2}} \right)}^{2}}}$. Use the definition of $q_i$, i.e. equation (\ref{main equation of mean interaction}), we have $C_i=\frac{\dot{q}_{i}^{2}}{2}-\frac{2}{{{q}_{i}}}-\frac{{{q}_{i}}}{{{\left( {{{\bar{q}}}_{1}}\text{+}{{{\bar{q}}}_{2}} \right)}^{2}}}$ and integrating this from 0 to 1 we can get $C_i=\frac{{{\left\| {{{\dot{q}}}_{i}} \right\|}^{2}}}{2}-\int_{0}^{1}{\frac{2}{{{q}_{i}}\left( s \right)}ds}-\frac{{{{\bar{q}}}_{i}}}{{{\left( {{{\bar{q}}}_{1}}\text{+}{{{\bar{q}}}_{2}} \right)}^{2}}}$. This is exactly the constant $c_i$ in (\ref{the equation of c_i}).
	
	And we modulo $q_i$ of ${{\ddot{q}}_{i}}{{q}_{i}}=C_i-\frac{\dot{q}_{i}^{2}}{2}+\frac{2{{q}_{i}}}{{{\left( {{{\bar{q}}}_{1}}\text{+}{{{\bar{q}}}_{2}} \right)}^{2}}}$, we get (\ref{the equation of q_i}). Because $z_i(\tau)^2:=q_i(t_{q_i}(\tau))$ as in (\ref{inverse transformation}) and  $z_i$ changed its sign at each zero, then we can finally get $z_i$ is smooth and satisfying equation (\ref{uncoupled ODE}).

\end{proof}

Therefore, we can get Theorem \ref{Main1}.

Next, we discuss situations when generalized solutions have an odd number of zeros.
%给出奇数个奇点时候的定理
\begin{thm}{(Cieliebak, Frauenfelder and Volkov \cite{cieliebak2022variational})}\label{theorem 3.5}
	Under the Levi-Civita transformation (\ref{ Levi-Civita transformations2}) with time change (\ref{time change2}), critical points $z_i$ of the functionals $\bar{\mathscr{D}}(z_i):{H}^{1}_{i,twist}(S^1,\mathbb{R})\backslash \left\{0\right\}\to \mathbb{R}$ on twisted loops are in 2-to-1 correspondence with generalized solutions $q_i:S^1\to \mathbb{R}_{\ge 0}$ of (\ref{main equation of mean interaction}) having an odd number of zeros, where $i=1,2$, ${H}^{1}_{i,twist}(S^1,\mathbb{R}):=\left\{ \left. {{z}_{i}}\in {{H}^{l}}\left( {\mathbb{R}}/{2\mathbb{Z}}\;,\mathbb{R} \right) \right|{{z}_{i}}\left( \tau +1 \right)=-{{z}_{i}}\left( \tau  \right)\text{ for all }\tau  \right\}$, $l\in {{\mathbb{N}}_{0}}$, and $\bar{\mathscr{D}}(z_i)$ have same form of expression as in (\ref{Lagrangian}).
\end{thm}

Now, for the $e^{-}Z^{2+}e^{-}$ Helium, there are four configurations. Without losing generality, let's assume the left electron undergoes an odd number of collisions and the right one undergoes an even number of collisions, others can be similar.

Then we can introduce for each $l\in {{\mathbb{N}}_{0}}$, the Hilbert space of twisted loops $${H}^{l}_{twist}(S^1,\mathbb{R}):=\left\{ \left. z=(z_1,z_2)\in {{H}^{l}}\left( {\mathbb{R}}/{2\mathbb{Z}}\;,\mathbb{R} \right) \right|{{z}_{1}}\left( \tau +1 \right)=-{{z}_{1}}\left( \tau  \right),{{z}_{2}}\left( \tau +1 \right)=-{{z}_{2}}\left( \tau  \right)\text{ for all }\tau  \right\}$$ with the inner product $$\left\langle z,v \right\rangle :=\sum\limits_{i=1}^{2}{\int_{0}^{1}{z\left( \tau  \right)v\left( \tau  \right)}}d\tau =\frac{1}{2}\sum\limits_{i=1}^{2}{\int_{0}^{2}{z\left( \tau  \right)v\left( \tau  \right)}}d\tau .$$
%这里也要模掉torus action来定义吗？

Then we define $$\bar{\mathscr{H}}_{av}^{1}:=\left\{ \left. z=\left( {{z}_{1}},{{z}_{2}} \right)\in H_{twist}^{1}\left( {{S}^{1}},{{\mathbb{R}}^{2}} \right) \right|\left\| {{z}_{1}} \right\|>0,\left\| {{z}_{2}} \right\|>0 \right\}.$$

Therefore, from Theorem \ref{theorem 3.5} and in analogy with Theorem \ref{Main1}, we can obtain the twisted version of Theorem \ref{Main1}.

\begin{thm}
	Under the Levi-Civita transformations (\ref{ Levi-Civita transformations2}) with time change (\ref{time change2}), critical points $(z_1,z_2)$ of the action functional $\bar{\mathscr{B}_{av}}:\tilde{\mathscr{H}}_{av}^{1}\to \mathbb{R}$ on twisted loops are in 4-to-1 correspondence with generalized solutions $(q_1,q_2)$ of (\ref{main equation of mean interaction}) with the left electron having an odd number of zeros and the right electron having an even number of zeros, where $\tilde{\mathscr{B}_{av}}$ has the same form of expression as in (\ref{action functional of mean helium}).
\end{thm}

\subsection{Hamiltonian Formulation}
  We apply an abstract Legendre transform , which is established from \cite{cieliebak2022variational}, to the $e^{-}Z^{2+}e^{-}$ configuration in the Helium atom. We abbreviate $H^l:=H^l(S^1, {\mathbb{R}}^n)$, $l\in {\mathbb{N}}_0$, the derivative of $q\in H^1$ is denoted by $\dot{q}$. 
  
  Assume now we have an open subset $\mathcal{U}^1\subset H^1$ and a Lagrange function $$\mathscr{L}: \mathcal{U}^1 \times H^0 \to \mathbb{R}, (q,v)\mapsto \mathscr{L}(q,v).$$ 
  
  If $\mathscr{L}$ is of class of $C^1$ and there exists a continuous map: $$\nabla \mathscr{L}=(\nabla_{1} \mathscr{L}, \nabla_{2} \mathscr{L}): \mathcal{U}^1 \times H^0 \to H^0 \times H^0,$$ which is uniquely defined by the conditions that for all $\omega\in H^1$, $\left\langle {\nabla_{i} \mathscr{L}(q,v), \omega} \right\rangle=D_{i}\mathscr{L}(q,v)\omega$. Where $\left\langle{,}\right\rangle$ is the $L^2$-inner product and $D_{i}\mathscr{L}$ denotes the derivatives with regard to the $i$-th variable. Then we say $\mathscr{L}$ possesses a continuous $L^2$-gradient.
  
  To such $\mathscr{L}$, we associate its Lagrangian action: $\mathcal{S}_{\mathscr{L}}:\mathcal{U}^1\to \mathbb{R}, q\mapsto \mathscr{L}(q,\dot{q}).$ Through computations, $q\in \mathcal{U}^1$ is a critical point of $\mathcal{S}_{\mathscr{L}}$ if and only if $\nabla_{2} \mathscr{L}(q,\dot{q})\in H^1$ and there is a following Euler-Lagrange equation holds: $$\frac{d}{dt}\nabla_{2} \mathscr{L}(q,\dot{q})=\nabla_{1} \mathscr{L}(q,\dot{q}).$$
  
  We add a further condition on $\mathscr{L}$. 
  
  (L): There exists a differentiable map $W:\mathcal{U}^1 \times H^0 \to H^0 \times H^0, (q,p)\mapsto W(q,p)$ such that for each $q\in \mathcal{U}^1$, the map $H^0\to H^0, v\mapsto \nabla_{2} \mathscr{L}(q,v)$ is a homeomorphism with inverse $p\mapsto W(q,p)$.
 Particularly, we have $\nabla_{2} \mathscr{L}(q,W(q,p))=p$. 
 
 We give a fibrewise Legendre transform associated to $\mathscr{L}$: $$\mathscr{H}:\mathcal{U}^1 \times H^0 \to \mathbb{R}, \mathscr{H}(q,p):=\left\langle{p, W(q,p)}\right\rangle-\mathscr{L}(q, W(q,p)).$$Then $\mathscr{H}$ has a continuous $L^2$-gradient related to $\mathscr{L}$ by \begin{equation*}
  	\begin{aligned}
  		&\nabla_{1} \mathscr{H}(q,p)=-\nabla_{1} \mathscr{L}(q, W(q,p)),\\
  		&\nabla_{2} \mathscr{H}(q,p)=W(q,p).\\
  	\end{aligned}
  \end{equation*}

   We can associate any Hamilton function $\mathscr{H}:\mathcal{U}^1 \times H^0 \to \mathbb{R}$ its Hamiltonian action $$ \mathscr{A}_{\mathscr{H}}:\mathcal{U}^1 \times H^0 \to \mathbb{R}, \mathscr{A}_{\mathscr{H}}(q,p):=\left\langle{p,\dot{q}}\right\rangle-\mathscr{H}(q,p)$$ with continuous $L^2$-gradient. 
   
   By calculating, $q\in \mathcal{U}^1$ is a critical point $\mathscr{A}_{\mathscr{H}}$ if and only if $p\in H^1$ and there are following Hamilton equations holds:
   \begin{equation*}
   	\begin{aligned}
   		&\dot{p}=-\nabla_{1} \mathscr{H}(q,p),\\
   		&\dot{q}=\nabla_{2} \mathscr{H}(q,p).\\
   	\end{aligned}
   \end{equation*}
   
   \begin{proposition}({Cieliebak, Frauenfelder and Volkov\cite{cieliebak2022variational}})\label{proposition2}
   	Let $\mathscr{L}:\mathcal{U}^1 \times H^0 \to \mathbb{R}$ be a Lagrange function with continuous gradient  satisfying condition (L) and $\mathscr{H}:\mathcal{U}^1 \times H^0 \to \mathbb{R}$ its fibrewise Legendre transform. Then the assignments $(q, p)\mapsto q$ and $q\mapsto (q, p=\nabla_{2} \mathscr{L}(q,\dot{q}))$ define a 1-to-1 correspondence between critical points $(q,p)$ of $\mathscr{A}_{\mathscr{H}}$ and critical points $q$ of $S_{\mathscr{L}}$.
   \end{proposition}
   
   Now in the $e^{-}Z^{2+}e^{-}$ Helium atom with mean interaction, the function $\mathscr{B}_{av}$ defined in (\ref{Bav}) is the Lagrangian action $\mathcal{S}_{{\mathscr{L}_{av}}}$ associated to the Lagrange function $$\mathscr{L}_{av}:\mathscr{H}^{1}_{av}\times H^{0}(S^1,\mathbb{R}^2)\to \mathbb{R},  \mathscr{L}_{av}(z,w)=\mathscr{L}(z_1,w_1)+\mathscr{L}(z_2,w_2)+\mathscr{A}(z_1,z_2),$$ where $\mathscr{H}^{1}_{av}$ and $\mathscr{A}$ are defined in Section \ref{H1av} and (\ref{mean interaction}), $\mathscr{L}$ is the Kepler Lagrangian with charge $N=2$, i.e. $$\mathscr{L}_{av}(z_i,w_i)=2{{\left\| z_i \right\|}^{2}}{{\left\| w_i \right\|}^{2}}+\frac{2}{{{\left\| z_i \right\|}^{2}}}(i=1,2).$$
   
   In section \ref{section3.2}, from the computation of the differential of $\mathscr{B}_{av}$, we know that $\mathscr{L}_{av}$ has a continuous $L^2$-gradient. The interaction term $\mathscr{A}$ is independent with $w_i$, so the associated momenta $\eta_{i}$ are given by $$\eta_{i}=\nabla_{2}\mathscr{L}_{av}(z_i, w_i)=4{{\left\| z_i \right\|}^{2}}(w_i)^2,$$ which can be solved for $w_i$ as $$w_i=\frac{\eta_{i}}{4{{\left\| z_i \right\|}^{2}}}=W(z_i, \eta_{i}).$$

   Then the associated Hamilton function becomes \begin{equation*}
   	\begin{aligned}
   		\mathscr{H}_{av}(z,\eta)&=\mathscr{H}(z_1,\eta_{1})+ \mathscr{H}(z_2,\eta_{2})-\mathscr{A}(z_1,z_2)\\
   		&=\sum\limits_{i=1}^{2}{\left( \frac{{{\left\| {{\eta }_{i}} \right\|}^{2}}}{8{{\left\| {{z}_{i}} \right\|}^{2}}}-\frac{2}{{{\left\| {{z}_{i}} \right\|}^{2}}} \right)}+\frac{{{\left\| {{z}_{1}} \right\|}^{2}}{{\left\| {{z}_{2}} \right\|}^{2}}}{{{\left\| z_{1}^{2} \right\|}^{2}}{{\left\| {{z}_{2}} \right\|}^{2}}+{{\left\| z_{2}^{2} \right\|}^{2}}{{\left\| {{z}_{1}} \right\|}^{2}}}\\
   	\end{aligned}
   \end{equation*}
    with Hamilton action \begin{equation*}
    	\begin{aligned}
    		\mathscr{A}_{\mathscr{H}_{av}}(z,\eta)&=\left\langle {{\eta }_{1}},{{z}'_{1}} \right\rangle +\left\langle {{\eta }_{2}},{{z}'_{2}} \right\rangle-\mathscr{H}_{av}(z,\eta)\\
    		&=\sum\limits_{i=1}^{2}{\left\langle {{\eta }_{i}},{{z}'_{i}} \right\rangle-\left( \frac{{{\left\| {{\eta }_{i}} \right\|}^{2}}}{8{{\left\| {{z}_{i}} \right\|}^{2}}}+\frac{2}{{{\left\| {{z}_{i}} \right\|}^{2}}} \right)}-\frac{{{\left\| {{z}_{1}} \right\|}^{2}}{{\left\| {{z}_{2}} \right\|}^{2}}}{{{\left\| z_{1}^{2} \right\|}^{2}}{{\left\| {{z}_{2}} \right\|}^{2}}+{{\left\| z_{2}^{2} \right\|}^{2}}{{\left\| {{z}_{1}} \right\|}^{2}}}.\\
    	\end{aligned}
    \end{equation*}
    
    By Proposition \ref{proposition2}, critical points of $\mathscr{A}_{\mathscr{H}_{av}}$ are in 1-to-1 correspondence to critical points of $\mathscr{B}_{av}$.
    
	\section{Periodic orbits in mean $e^{-}Z^{2+}e^{-}$ Helium}
	%Classify the periodic orbits by their time period $\sigma$.（acoording to Urs's article)
	
	%Goal: Prove that there are one-to-one correspondence between positive rational numbers and solutions of mean interaction equation of Helium
	
	For $\sigma>0$ we introduce the vector space $${{V}_{\sigma }}:=\left\{ q\in {{C}^{0}}\left( \left[ 0,\sigma  \right],\left[ 0,\infty  \right) \right)\cap {{C}^{\infty }}\left( \left[ 0,\sigma  \right),\left[ 0,\infty  \right) \right):\dot{q}\left( 0 \right)=0,q\left( \sigma  \right)=0, \forall t\in \left[0,\sigma\right), q(t)>0 \right\},$$ where $\dot{q}=\frac{dq}{dt}$ denotes the derivative of $q$. We refer the space ${{V}_{\sigma }}$ as the space of free falls of period $\sigma$.
	%(Sundman Integral)
	For $q\in {{V}_{\sigma }}$, we abbreviate by $\bar{q}=\frac{1}{\sigma}\int_{0}^{\sigma}q(t)dt$ the mean value of $q$.
	
	Given $\sigma_{1},\sigma_{2}>0$, we first look for solutions $(q_{1}^{\sigma_{1}},q_{2}^{\sigma_{2}})\in {{V}_{\sigma_{1} }}\times {{V}_{\sigma_{2} }}$ solving the problem (\ref{main equation of mean interaction}), i.e.
	\begin{equation} \label{updates mean equation of mean interaction}
		\begin{aligned}
			&{\ddot{q}_{1}}^{\sigma_{1}}=-\frac{2}{\left(q_{1}^{\sigma_{1}}\right)^2}+\frac{1}{ \left({\bar{q}_{1}}^{\sigma_{1}}+{\bar{q}_{2}}^{\sigma_{2}}\right)^2} \text{  }\text{  } t\in \left[0,\sigma_{1} \right),\\
			&{\ddot{q}_{2}}^{\sigma_{2}}=-\frac{2}{\left(q_{2}^{\sigma_{2}}\right)^2}+\frac{1}{ \left({\bar{q}_{1}}^{\sigma_{1}}+{\bar{q}_{2}}^{\sigma_{2}}\right)^2} \text{  }\text{  }  t\in \left[0,\sigma_{2} \right).
		\end{aligned}
	\end{equation}

%Denote $C:=\frac{1}{ \left(\bar{q}_{1}+\bar{q}_{2}\right)^2}\ge 0$
\begin{lemma}
	For $\sigma>0$ and $m\ge 0$, when $m>0$, $q_{m,\sigma}(0)\le \sqrt{\frac{2}{m}}$. $q_{m,\sigma}$ also satisfies $q_{m,\sigma}(\sigma)=0$, $\dot{q}_{m,\sigma}(0)=0$. Then there exists a unique $q_{m,\sigma}\in V_{\sigma}$ solving the ODE
	\begin{equation}\label{yibanODE}
		{\ddot{q}_{m,\sigma}}(t)=-\frac{2}{{q_{m,\sigma}}^2}+m, t\in \left[0,\sigma\right).
	\end{equation} 
\end{lemma}
\begin{proof}
	Let $p_{m,\sigma}=\dot{q}_{m,\sigma}$, for $t\in \left[0,\sigma\right)$, we rewrite (\ref{yibanODE}) as \begin{equation*}
	\left\{	\begin{aligned}
			&\dot{q}_{m,\sigma}=p_{m,\sigma},\\
			&\dot{p}_{m,\sigma}=-\frac{2}{{q_{m,\sigma}}^2}+m.
		\end{aligned} \right.
	\end{equation*}

Denote $F(q_{m,\sigma},p_{m,\sigma})=(p_{m,\sigma},-\frac{2}{{q_{m,\sigma}}^2}+m)^{T}$, we have system $(q_{m,\sigma},p_{m,\sigma})^{T}=F(q_{m,\sigma},p_{m,\sigma})$ is defined on $D=\left\{(q_{m,\sigma},p_{m,\sigma})\in \mathbb{R}^2 | q>0\right\}$. The energy $E=E(0)=\frac{\dot{q}_{m,\sigma}(t)^2}{2}-\frac{2}{q_{m,\sigma}(t)}-mq_{m,\sigma}(t)=-\frac{2}{q_{m,\sigma}(0)}-mq_{m,\sigma}(0)$, the velocity satisfies $p_{m,\sigma}^2=2\left( \frac{2}{q_{m,\sigma}}-\frac{2}{q_{m,\sigma}\left(0\right)}+mq_{m,\sigma}-m{q_{m,\sigma}\left(0\right)} \right)$.

Since $\dot{p}_{m,\sigma}(0)=-\frac{2}{{q_{m,\sigma}}^2}+m\le 0$ and $\dot{q}_{m,\sigma}(0)=0$, $p_{m,\sigma}(t)\le0$. If there exists a $t_1\in\left(0,\sigma\right)$, s.t. $p_{m,\sigma}(t_1)=0$, then $q_{m,\sigma}(t_1)=q_{m,\sigma}(0)$, $q_{m,\sigma}\equiv q_{m,\sigma}(0)$, which contradicts the condition $q_{m,\sigma}(\sigma)=0$. So $\forall t\in\left(0,\sigma\right)$, $p_{m,\sigma}(t)<0$.

Since $F(q_{m,\sigma},p_{m,\sigma})$ is continuously differentiable on $D$, it satisfies Lipschitz condition for any compact subset of $D$. According to the theorem on the existence and uniqueness of solutions to ODE, there exists a unique maximum interval $I=\left[0, \beta \right)$, s.t. $(q_{m,\sigma},p_{m,\sigma})^{T}$ solves (\ref{yibanODE}) and it is unique.

From $q_{m,\sigma}(\sigma)=0$, then $\beta\le \sigma$. Assume $\beta<\sigma$, due to $q_{m,\sigma}(t)$ is strictly monotonically decreasing and bounded by a lower bound 0, $\underset{t\to \beta^{-} }{\mathop{\lim }}\,{q_{m,\sigma}}\left( t\right)=q_{m,\sigma}(\beta)\ge 0$ exists.

If $q_{m,\sigma}(\beta)>0$, $\forall t\in \left[0, \beta \right)$, $q_{m,\sigma}(t)\ge q_{m,\sigma}(\beta)>0$, there exists a finite number $M>0$, s.t.$\frac{(p_{m,\sigma})^2}{2}\le \frac{2}{q_{m,\sigma}(\beta)}+mq_{m,\sigma}(\beta)-\frac{2}{q_{m,\sigma}(0)}-mq_{m,\sigma}(0)\le M$. Then the curve of solution locates in a compact subset $K=\left[q_{m,\sigma}(\beta),q_{m,\sigma}(0)\right]\times \left[\sqrt{2M},\sqrt{2M}\right]$$\subseteq D$. From the theory of ODE, we know the solution can be extended to a bigger interval, contradicting to the definition of maximal interval. 

So $q_{m,\sigma}(\beta)=0$， then $\beta=\sigma$. There exists a unique $q_{m,\sigma}(t)\in V_{\sigma}(t\in\left[0,\sigma\right))$ solving the ODE (\ref{yibanODE}).
\end{proof}

Now, given a solution $q_{m,\sigma}$ of (\ref{yibanODE}), we consider the function $$f_{\sigma}:\left[0,\infty\right) \to \left[0,\infty\right), m\mapsto \sqrt{m}\cdot {{\bar{q}}_{m,\sigma }}.$$

\begin{proposition}
	For any $\sigma>0$, the function $f_{\sigma}$ is strictly monotonically increasing such that $\underset{m\to \infty }{\mathop{\lim }}\,{{f}_{\sigma }}\left( m \right)=\sqrt{2}$.
\end{proposition}

\begin{proof}
We abbreviate $q=q_{m,\sigma}$, we can see the energy $E:=\frac{\dot{q}(t)^2}{2}-\frac{2}{q(t)}-mq(t)$, since $\dot{q}(0)=0$ and the law of conservation of energy, we have $E=E(0)=-\frac{2}{q(0)}-mq(0)$.

And we can express the velocity with the help of the energy by the formula $\dot{q}=-\sqrt{2\left( \frac{2}{q}-\frac{2}{q_{0}}+mq-m{q_{0}} \right)}$, where $q_0:=q(0)\in\left(0,\sqrt{\frac{2}{m}}\right]$ represents general initial position, . The sign of the velocity is negative because both two electrons move towards the nucleus.

Using $q(\sigma)=0$, we can compute $\sigma$ as follows:
\begin{equation}\label{the formula of time}
	\begin{aligned}
		\sigma=\int_{0}^{\sigma }{dt}=\int_{q_0}^{0}\frac{dq}{\dot{q}}=\int_{0}^{q_0}\frac{dq}{\sqrt{2\left( \frac{2}{q}-\frac{2}{q_{0}}+mq-m{q_{0}} \right)}}&=\int_{0}^{1}{\frac{{q_0}du}{\sqrt{2\left( \frac{2}{{q_0}}\left( \frac{1-u}{u} \right)+m{q_0}\left( u-1 \right) \right)}}} \\ 
		&=\int_{0}^{1}{\frac{q_{0}^{\frac{3}{2}}\cdot \sqrt{u}du}{\sqrt{2\left( 1-u \right)\left( 2-mq_{0}^{2}\cdot u \right)}}} \\ 
		&=\frac{q_{0}^{\frac{3}{2}}}{2}\int_{0}^{1}{\frac{\sqrt{u}du}{\sqrt{\left( 1-u \right)\left( 1-{k}u \right)}}}. \\ 
	\end{aligned}	
\end{equation}
Where we change the variable $u=\frac{q}{{q_0}}$ and $k=\frac{mq_{0}^{2}}{2}$. We denote $f(k):=\int_{0}^{1}{\frac{\sqrt{u}du}{\sqrt{\left( 1-u \right)\left( 1-ku \right)}}}$, $k\in \left[0,1\right)$.

Then we can also obtain the formula of $\bar{q}$,
\begin{equation}
	\begin{aligned}
		\bar{q}=\frac{1}{{\sigma }}\int_{0}^{{\sigma}}{q}\left( t \right)dt&=\frac{1}{q_{0}^{\frac{3}{2}}f\left( k \right)}\int_{{q_0}}^{0}{\frac{{q}d{q}}{{\dot{q}}\left( t \right)} }\\ 
		&=\frac{1}{q_{0}^{\frac{3}{2}}f\left( k\right)}\text{  }\int_{0}^{{q_0}}{\frac{{q}d{q}}{\sqrt{2\left( \frac{2}{q\left( t \right)}-\frac{2}{{q_0}}+m\left( q-{q_0} \right) \right)}}} \\ 
		& =\frac{1}{2q_{0}^{\frac{3}{2}}f\left( k \right)}\int_{0}^{1}{\frac{q_{0}^{\frac{5}{2}}\cdot {{u}^{\frac{3}{2}}}du}{\sqrt{\left( 1-u \right)\left( 1-ku \right)}}}. \\ 
	\end{aligned}
\end{equation}

We denote $g(k):=\int_{0}^{1}{\frac{{{u}^{\frac{3}{2}}}du}{\sqrt{\left( 1-u \right)\left( 1-ku \right)}}}$, $k=\frac{mq_{0}^{2}}{2} \in \left[0,1\right)$, and $h(k):=\frac{g(k)}{f(k)}$, $k\in \left[0,1\right)$, then we can get $${\bar{q}}=\frac{2q_{0}^{\frac{5}{2}}\cdot g\left( k \right)}{2q_{0}^{\frac{3}{2}}\cdot f\left( k \right)}={q_0}\cdot h\left( k \right),$$ and $${{f}_{\sigma }}\left( m \right)=\sqrt{m}\cdot {{\bar{q}}_{m,\sigma }}={{q}_{0}}\sqrt{m}\cdot h\left( k \right)=\sqrt{2k}\cdot h(k).$$

We first calculate the first derivative of $f(k)$, $g(k)$ and $h(k)$ separately. From the definitions above, we know the ranges of these three functions are all $\mathbb{R}_{>0}$.

$${h}'\left( k \right)=\frac{{g}'\left( k \right)f\left( k \right)-g\left( k \right){f}'\left( k \right)}{{{\left( f\left( k \right) \right)}^{2}}},$$

$${f}'\left( k\right)=\frac{d}{dk}\int_{0}^{1}{\frac{\sqrt{u}du}{\sqrt{\left( 1-u \right)\left( 1-ku \right)}}}=\frac{1}{2}\int_{0}^{1}{\frac{{{u}^{\frac{3}{2}}}}{\sqrt{\left( 1-u \right){{\left( 1-ku \right)}^{3}}}}du},$$
%\begin{equation*}
%	\begin{aligned}
%		{f}'\left( k\right)&=\frac{d}{dk}\int_{0}^{1}{\frac{\sqrt{u}du}{\sqrt{\left( 1-u \right)\left( 1-ku \right)}}} \\ 
%		&=\frac{1}{2}\int_{0}^{1}{\frac{{{u}^{\frac{3}{2}}}}{\sqrt{\left( 1-u \right){{\left( 1-ku \right)}^{3}}}}du}, \\ 
%	\end{aligned}
%\end{equation*}
$${g}'\left( k \right) =\frac{d}{dk}\int_{0}^{1}{\frac{{{u}^{\frac{3}{2}}}du}{\sqrt{\left( 1-u \right)\left( 1-ku \right)}}}=\frac{1}{2}\int_{0}^{1}{\frac{{{u}^{\frac{5}{2}}}}{\sqrt{\left( 1-u \right){{\left( 1-ku \right)}^{3}}}}du},$$
%\begin{equation*}
%	\begin{aligned}
%		{g}'\left( k \right) &=\frac{d}{dk}\int_{0}^{1}{\frac{{{u}^{\frac{3}{2}}}du}{\sqrt{\left( 1-u \right)\left( 1-ku \right)}}} \\ 
%		&=\frac{1}{2}\int_{0}^{1}{\frac{{{u}^{\frac{5}{2}}}}{\sqrt{\left( 1-u \right){{\left( 1-ku \right)}^{3}}}}du}, \\ 
%	\end{aligned}   		
%\end{equation*}

\begin{equation*}
	\begin{aligned}
		&{g}'\left( k \right)f\left( k \right)-g\left( k \right){f}'\left( k \right) \\ 
		& =\frac{1}{2}\int_{0}^{1}{\frac{{{udu}^{\frac{5}{2}}}}{\sqrt{\left( 1-u \right){{\left( 1-ku \right)}^{3}}}}}\int_{0}^{1}{\frac{\sqrt{v}dv}{\sqrt{\left( 1-v \right)\left( 1-kv \right)}}}-\frac{1}{2}\int_{0}^{1}{\frac{{{u}^{\frac{3}{2}}}du}{\sqrt{\left( 1-u \right)\left( 1-ku \right)}}}\int_{0}^{1}{\frac{{{v}^{\frac{3}{2}}}dv}{\sqrt{\left( 1-v \right){{\left( 1-kv \right)}^{3}}}}} \\ 
		& =\frac{1}{2}\iint\limits_{{{\left[ 0,1 \right]}^{2}}}{\frac{{{u}^{\frac{3}{2}}}\sqrt{v}}{\sqrt{\left( 1-u \right)\left( 1-v \right){{\left( 1-ku \right)}^{3}}{{\left( 1-kv \right)}^{3}}}}}\left( u\left( 1-kv \right)-v\left( 1-ku \right) \right)dudv \\ 
		& =\frac{1}{2}\iint\limits_{{{\left[ 0,1 \right]}^{2}}}{\frac{{{u}^{\frac{3}{2}}}\sqrt{v}}{\sqrt{\left( 1-u \right)\left( 1-v \right){{\left( 1-ku \right)}^{3}}{{\left( 1-kv \right)}^{3}}}}}\left( u-v \right)dudv \\ 
		& =\frac{1}{2}\left( \iint\limits_{u>v}{\frac{{{u}^{\frac{3}{2}}}\sqrt{v}}{\sqrt{\left( 1-u \right)\left( 1-v \right){{\left( 1-ku \right)}^{3}}{{\left( 1-kv \right)}^{3}}}}}\left( u-v \right)dudv+\iint\limits_{u<v}{\frac{{{u}^{\frac{3}{2}}}\sqrt{v}}{\sqrt{\left( 1-u \right)\left( 1-v \right){{\left( 1-ku \right)}^{3}}{{\left( 1-kv \right)}^{3}}}}}\left( u-v \right)dudv \right) \\ 
		& =\frac{1}{2}\left( \iint\limits_{u>v}{\frac{{{u}^{\frac{3}{2}}}\sqrt{v}\left( u-v \right)dudv}{\sqrt{\left( 1-u \right)\left( 1-v \right){{\left( 1-ku \right)}^{3}}{{\left( 1-kv \right)}^{3}}}}}+\iint\limits_{v'<u'}{\frac{v{{'}^{\frac{3}{2}}}\sqrt{u'}\left( v'-u' \right)dv'du'}{\sqrt{\left( 1-v' \right)\left( 1-u' \right){{\left( 1-kv' \right)}^{3}}{{\left( 1-ku' \right)}^{3}}}}} \right) \\ 
		& =\frac{1}{2}\left( \iint\limits_{u>v}{\frac{{{u}^{\frac{3}{2}}}\sqrt{v}}{\sqrt{\left( 1-u \right)\left( 1-v \right){{\left( 1-ku \right)}^{3}}{{\left( 1-kv \right)}^{3}}}}}\left( u-v \right)dudv+\iint\limits_{u>v}{\frac{{{v}^{\frac{3}{2}}}\sqrt{u}}{\sqrt{\left( 1-u \right)\left( 1-v \right){{\left( 1-ku \right)}^{3}}{{\left( 1-kv \right)}^{3}}}}}\left( v-u \right)dudv \right) \\ 
		& =\frac{1}{2}\iint\limits_{u>v}{\frac{\sqrt{u}\sqrt{v}}{\sqrt{\left( 1-u \right)\left( 1-v \right){{\left( 1-ku \right)}^{3}}{{\left( 1-kv \right)}^{3}}}}}{{\left( u-v \right)}^{2}}dudv .\\ 
	\end{aligned}
\end{equation*}

For any $k\in \left[0,1\right)$, ${g}'\left( k \right)f\left( k \right)-g\left( k \right){f}'\left( k \right)>0$, so for any $k\in \left[0,1\right)$, ${h}'\left( k \right)>0$.

Then we calculate the derivative of $f_{\sigma}$. 

Differentiating (\ref{the formula of time}) we obtain $0=\frac{3}{4}q_{0}^{\frac{1}{2}}{{q}'_0}f\left( k \right)+\frac{1}{2}q_{0}^{\frac{3}{2}}\left( \frac{1}{2}q_{0}^{2}+m{{q}_{0}}{q}' \right){f}'\left( k \right)$. Since the velocity ${q}'<0$, this implies that $$k'(m)=\left( \frac{1}{2}q_{0}^{2}+m{{q}_{0}}{q}' \right)>0,$$ then $${{f}'_{\sigma }}\left( m \right)=\frac{{k}'\left( m \right)h\left( k \right)}{\sqrt{2k}}+\sqrt{2k}{h}'\left( k \right){k}'\left( m \right)>0.$$

For the limit, we first denote the primitive functions of $g(k)$ and $f(k)$ by $G(u)$ and $F(u)$ separately. When $k\to 1$, $G(u)=\int{\frac{{{u}^{\frac{3}{2}}}}{\left( 1-u \right)}}du=-\frac{2{{u}^{\frac{3}{2}}}}{3}-2\sqrt{u}+\ln \left| \frac{1+\sqrt{u}}{1-\sqrt{u}} \right|+C$ and $F(u)=\int{\frac{\sqrt{u}}{\left( 1-u \right)}}du=-2\sqrt{u}+\ln \left| \frac{1+\sqrt{u}}{1-\sqrt{u}} \right|+C$ are both divergence when $u\to 1^{-}$. So we rewrite $h(k)$ in the following:
$$	h\left( k \right)=\frac{g\left( k \right)}{f\left( k \right)}=\frac{\int_{0}^{1}{\frac{{{u}^{\frac{3}{2}}}}{\sqrt{\left( 1-u \right)\left( 1-ku \right)}}}du}{\int_{0}^{1}{\frac{\sqrt{u}}{\sqrt{\left( 1-u \right)\left( 1-ku \right)}}}du}=\frac{\int_{0}^{1-\delta }{\frac{{{u}^{\frac{3}{2}}}}{\sqrt{\left( 1-u \right)\left( 1-ku \right)}}}du+\int_{1-\delta }^{1}{\frac{{{u}^{\frac{3}{2}}}}{\sqrt{\left( 1-u \right)\left( 1-ku \right)}}}du}{\int_{0}^{1-\delta }{\frac{\sqrt{u}}{\sqrt{\left( 1-u \right)\left( 1-ku \right)}}}du+\int_{1-\delta }^{1}{\frac{\sqrt{u}}{\sqrt{\left( 1-u \right)\left( 1-ku \right)}}}du},$$
where $\underset{k\to 1}{\mathop{\lim }}\,\int_{0}^{1-\delta }{\frac{{{u}^{\frac{3}{2}}}}{\sqrt{\left( 1-u \right)\left( 1-ku \right)}}}du$ and $\underset{k\to 1}{\mathop{\lim }}\,\int_{0}^{1-\delta }{\frac{\sqrt{u}}{\sqrt{\left( 1-u \right)\left( 1-ku \right)}}}du$  are finite values. From the expressions of $G(u)$ and $F(u)$, we can see $\underset{k\to 1}{\mathop{\lim }}\,\int_{1-\delta }^{1}{\frac{{{u}^{\frac{3}{2}}}}{\sqrt{\left( 1-u \right)\left( 1-ku \right)}}}du$ and $\underset{k\to 1}{\mathop{\lim }}\,\int_{1-\delta }^{1}{\frac{\sqrt{u}}{\sqrt{\left( 1-u \right)\left( 1-ku \right)}}}du$ both diverge only when $u\to 1^{-}$, and they diverge to the same degree.

In total, $\underset{k\to {{1}^{-}}}{\mathop{\lim }}\,h(k)=1$.

Since $\underset{m\to \infty }{\mathop{\lim }}\,{{q}_{m,\sigma }}(0)=0$, and from the formula (\ref{the formula of time}) of $\sigma$, we infer $k\to 1\left( m\to \infty  \right)$. So from the definition of $f_\sigma$, we get $\underset{m\to \infty }{\mathop{\lim }}\,{{f}_{\sigma }}\left( m \right)=\underset{m\to \infty }{\mathop{\lim }}\,\sqrt{2k}\cdot h\left( k \right)=\sqrt{2}$.

In the end, we prove this proposition.
\end{proof}
%We leave the details of the proof of this proposition in Appendix. 

\begin{lemma}\label{Lemma 4.1}
	For any pair $\sigma_{1},\sigma_{2}>0$ there exists a unique solution $(q_{1}^{\sigma_{1}},q_{2}^{\sigma_{2}})\in {{V}_{\sigma_{1} }}\times {{V}_{\sigma_{2} }}$ of problem  (\ref{updates mean equation of mean interaction}).
\end{lemma}

\begin{proof}
	We consider the function $f_{\sigma_{1},\sigma_{2}}:=f_{\sigma_{1}}+f_{\sigma_{2}}:\left[0,\infty\right) \to \left[0,\infty\right)$, by the proposition above, this function is strictly monotonically increasing and satisfies $\underset{m\to \infty }{\mathop{\lim }}\,{f_{\sigma_{1},\sigma_{2}}}\left( m \right)=2\sqrt{2}$. Moreover, we have $f_{\sigma_{1},\sigma_{2}}(0)=0$. We know $f_{\sigma}$ is strictly increasing, hence there exists a unique $m:=m_{\sigma_{1},\sigma_{2}}\in \left[0,\infty\right)$ with the property that $f_{\sigma_{1},\sigma_{2}}(m)=1$.
	
	We claim that $(q_{1}^{\sigma_{1}},q_{2}^{\sigma_{2}}):=(q_{m,\sigma_{1}}, q_{m,\sigma_{2}})
	\in V_{\sigma_{1}}\times V_{\sigma_{2}}$ solves problem (\ref{updates mean equation of mean interaction}). For that purpose we compute $$\bar{q}_{1}^{\sigma_{1}}+\bar{q}_{2}^{\sigma_{2}}=\frac{f_{\sigma_{1}}(m)}{\sqrt{m}}+\frac{f_{\sigma_{2}}(m)}{\sqrt{m}}=\frac{f_{\sigma_{1},\sigma_{2}}(m)}{\sqrt{m}}=\frac{1}{\sqrt{m}},$$ so that $\frac{1}{ \left(\bar{q}_{1}^{\sigma_{1}}+\bar{q}_{2}^{\sigma_{2}}\right)^2}=m$. Hence for $i\in {1,2}$, we obtain that for $t\in \left[0,\sigma_{i}\right)$, we have $${\ddot{q}_{m,\sigma_{i}}}(t)=-\frac{2}{({q_{i}^{\sigma_{i}}})^2}+m=-\frac{2}{({q_{i}^{\sigma_{i}}})^2}+\frac{1}{\left(\bar{q}_{1}^{\sigma_{1}}+\bar{q}_{2}^{\sigma_{2}}\right)^2}.$$
	
	Consequently, $(q_{1}^{\sigma_{1}},q_{2}^{\sigma_{2}})$ solves problem (\ref{updates mean equation of mean interaction}), so we get the existence. And by reading the argument backwards, we obtain from the uniqueness of $m=m_{\sigma_{1},\sigma_{2}}$ that this solution is unique as well.
\end{proof}

In the end, we can provide a complete proof of Theorem \ref{Main2}. 
%这里加上过渡，\sigma的比值为什么是有理数，为什么是正的，然后说明只考虑最小周期，也就是两个周期互质，通过前面z的构造可以知道z changes its sign at zero，所以1个单位时间内碰撞奇数次就有z（t+1)=-z(t)，碰撞偶数次就有z（t+1)=z(t)

%\begin{thm}\label{Theorem 4.2}
%	There is a one-to-one correspondence between $\mathbb{Q}_{+}$ and $\mathscr{C}_{\mathscr{B}_{av}}/{\mathbb{Z}}/{2\mathbb{Z}}\;\times {\mathbb{Z}}/{2\mathbb{Z}}\;$, in the sense that $z_i=-z_i(i=1,2)$ under the  ${\mathbb{Z}}/{2\mathbb{Z}}\;$ action.
%\end{thm}
%periodic solutions of problem  (\ref{uncoupled ODE}) under 
  \begin{proof}[Proof of Theorem \ref{Main2}]
  From (\ref{the formula of time}), we have $$\frac{{{\sigma }_{1}}}{{{\sigma }_{2}}}=\frac{2q_{1,0}^{\frac{3}{2}}\cdot f\left( {{k}_{1}} \right)}{2q_{2,0}^{\frac{3}{2}}\cdot f\left( {{k}_{2}} \right)}={{r}^{\frac{3}{2}}}\frac{f\left( {{k}_{1}} \right)}{f\left( {{k}_{2}} \right)}={{r}^{\frac{3}{2}}}\frac{f\left( {{r}^{2}}{{k}_{2}} \right)}{f\left( {{k}_{2}} \right)},$$ where $r=\frac{q_{1,0}^{\frac{3}{2}}}{q_{2,0}^{\frac{3}{2}}}\in \left(0,+\infty\right)$ and ${{k}_{1}}=\frac{1}{2}mq_{1,0}^{2}={{r}^{2}}\cdot \frac{1}{2}mq_{2,0}^{2}={{r}^{2}}{{k}_{2}}$, $q_{1,0}$ and $q_{2,0}$ are the initial position of $q_1$ and $q_2$ separately.
  
  Then we can compute $m$, using ${{q}_{1,0}}=r\cdot {{q}_{2,0}}$ and ${{k}_{1}}={{r}^{2}}{{k}_{2}}$, we have
  \begin{equation}
  	\begin{aligned}
  		 m&=\frac{1}{{{\left( {{{\bar{q}}}_{1}}+{{{\bar{q}}}_{2}} \right)}^{2}}} \\ 
  		&=\frac{1}{{{\left( {{q}_{1,0}}\cdot h\left( {{k}_{1}} \right)+{{q}_{2,0}}\cdot h\left( {{k}_{2}} \right) \right)}^{2}}} \\ 
  		&=\frac{1}{q_{2,0}^{2}{{\left( r\cdot h\left( {{r}^{2}}{{k}_{2}} \right)+h\left( {{k}_{2}} \right) \right)}^{2}}}. \\ 
  	\end{aligned}
  \end{equation}

  On the other hand, we know $m=\frac{2{{k}_{2}}}{q_{2,0}^{2}}$, so we have the quality formula $\frac{{{k}_{2}}}{q_{2,0}^{2}}=\frac{1}{2q_{2,0}^{2}{{\left( r\cdot h\left( {{r}^{2}}{{k}_{2}} \right)+h\left( {{k}_{2}} \right) \right)}^{2}}}$. Finally, we have this equation 
  \begin{equation}\label{original ziqia equation}
  	{{k}_{2}}=\frac{1}{2{{\left( r\cdot h\left( {{r}^{2}}{{k}_{2}} \right)+h\left( {{k}_{2}} \right) \right)}^{2}}}.
  \end{equation}
  
  In general, denote this equation by
  \begin{equation}\label{ziqia equation}
  	K\left( r,k \right)=k \left( r,k \right)=k-\frac{1}{2{{\left( rh\left( {{r}^{2}}k \right)+h\left( k \right) \right)}^{2}}}=0.
  \end{equation}
  
  Here we need ${{r}^{2}}k<1$ to make sure $h({r^2}{k})$ makes sense, then $k<\frac{1}{{{r}^{2}}}\Rightarrow k\in \left[0,k_{max}\right)$ and $r\in \left(0, +\infty\right)$. Where ${{k}_{\max }}=\min \left\{ 1,\frac{1}{{{r}^{2}}} \right\}=\left\{ \begin{aligned}
  	& 1,\text{     }\text{ }\text{ }r\le 1 \\ 
  	& \frac{1}{{{r}^{2}}},\text{ }r>1 \\ 
  \end{aligned} \right.$.
  
  \begin{lemma}
  	There exists a unique solution  $k=\kappa \left( r \right)$ solves the equation (\ref{ziqia equation}).
  \end{lemma}
   \begin{proof}
   	For every fixed $r$, ${{K}_{r}}\left( k \right)=K\left( r,k \right)$ is a function about $k\in \left[0,k_{max}\right)$. 
   	
   	Because for any $k\in (0,1)$, $h'(k)>0$, then we know the first derivative of $K_{r}(k)$ is $\frac{\partial {{K}_{r}}\left( k \right)}{\partial k}=1+\frac{\left( {{r}^{3}}{h}'\left( {{r}^{2}}k \right)+{h}'\left( k \right) \right)}{{{\left( rh\left( {{r}^{2}}k \right)+h\left( k \right) \right)}^{3}}}>0$. So ${{K}_{r}}\left( k \right)$ is strictly increasing in $\left[0,k_{max}\right)$.
   	
   	 About $\underset{k\to {{k}_{\max }}}{\mathop{\lim }}\,{{K}_{r}}\left( k \right)$, we discuss two cases.
   	
   	(1) When $r>1$, $\underset{k\to \frac{1}{{{r}^{2}}}}{\mathop{\lim }}\,{{K}_{r}}\left( k \right)=\frac{1}{{{r}^{2}}}-\frac{1}{2{{\left( r+h\left( \frac{2}{{{r}^{2}}} \right) \right)}^{2}}}=\frac{2{{\left( r+h\left( \frac{2}{{{r}^{2}}} \right) \right)}^{2}}-{{r}^{2}}}{2{{r}^{2}}{{\left( r+h\left( \frac{2}{{{r}^{2}}} \right) \right)}^{2}}}=\frac{{{r}^{2}}+2rh\left( \frac{2}{{{r}^{2}}} \right)+2h{{\left( \frac{2}{{{r}^{2}}} \right)}^{2}}}{2{{r}^{2}}{{\left( r+h\left( \frac{2}{{{r}^{2}}} \right) \right)}^{2}}}>0$;
%   	\begin{equation*}
%   		\begin{aligned}
%   			 \underset{k\to \frac{1}{{{r}^{2}}}}{\mathop{\lim }}\,{{K}_{r}}\left( k \right)&=\frac{1}{{{r}^{2}}}-\frac{1}{2{{\left( r+h\left( \frac{2}{{{r}^{2}}} \right) \right)}^{2}}} \\ 
%   			& =\frac{2{{\left( r+h\left( \frac{2}{{{r}^{2}}} \right) \right)}^{2}}-{{r}^{2}}}{2{{r}^{2}}{{\left( r+h\left( \frac{2}{{{r}^{2}}} \right) \right)}^{2}}} \\ 
%   			& =\frac{{{r}^{2}}+2rh\left( \frac{2}{{{r}^{2}}} \right)+2h{{\left( \frac{2}{{{r}^{2}}} \right)}^{2}}}{2{{r}^{2}}{{\left( r+h\left( \frac{2}{{{r}^{2}}} \right) \right)}^{2}}}>0;\\ 
%   		\end{aligned}
%   	\end{equation*}
   	
   	(2) When $r\le 1$, $\underset{k\to {{1}^{-}}}{\mathop{\lim }}\,{{K}_{r}}\left( k \right)=1-\frac{1}{2{{\left( rh\left( 2{{r}^{2}} \right)+\underset{k\to {{1}^{-}}}{\mathop{\lim }}\,h(k) \right)}^{2}}}=1-\frac{1}{2{{\left( rh\left( 2{{r}^{2}} \right)+1 \right)}^{2}}}\in \left( 0,\frac{1}{2} \right)$.
   	
   	Here, we have already proved $\underset{k\to {{1}^{-}}}{\mathop{\lim }}\,h(k)=1$ before.
    
   	On the other hand, ${{K}_{r}}\left( 0 \right)=-\frac{1}{2{{\left( r+1 \right)}^{2}}{{\left( h\left( 0 \right) \right)}^{2}}}<0$.
   	
   	In summary, we know that there exist a unique solution  $k=\kappa \left( r \right)$ solves the equation (\ref{ziqia equation}) in the domain $k\in \left[0,k_{max}\right)$. 
   \end{proof}

   \begin{remark}
   	We can rewrite the equation (\ref{original ziqia equation}) as  \begin{equation}\label{updates ziqia equation}
   		{{k}_{2}}=\frac{1}{{{\left( r\cdot h\left( {{r}^{2}}{\kappa(r)} \right)+h\left( {\kappa(r)} \right) \right)}^{2}}}.
   	\end{equation}
   \end{remark}

   Now we get the new expression $$\frac{{{\sigma }_{1}}}{{{\sigma }_{2}}}={{r}^{\frac{3}{2}}}\frac{f\left( {{r}^{2}}\cdot \kappa \left( r \right) \right)}{f\left( \kappa \left( r \right) \right)}.$$
  
  We denote $\Psi \left( r \right):={{r}^{\frac{3}{2}}}\frac{f\left( {{r}^{2}}\cdot \kappa \left( r \right) \right)}{f\left( \kappa \left( r \right) \right)}$, then we prove this function is a continuous, strictly monotonically increasing bijection on $\mathbb{R}_{+}$.
  
  \begin{lemma}
  	The function $\Psi \left( r \right)={{r}^{\frac{3}{2}}}\frac{f\left( {{r}^{2}}\cdot \kappa \left( r \right) \right)}{f\left( \kappa \left( r \right) \right)}:\mathbb{R}_{+}\to \mathbb{R}_{+}$ is a continuous and strictly monotonically increasing bijection.
  \end{lemma}
  \begin{proof}
  (1)Continuity
  
  From the definition of $\Psi \left( r \right)$, the continuity of $f$, the continuity of $\kappa(r)$ and the continuity of $r^{\frac{3}{2}}$, we see $\Psi \left( r \right)$ is continuous in $(0,+\infty)$.
  
  (2)Monotonicity
  
  Take the logarithm of both sides of this function $\Psi \left( r \right)$, we get
  
  $$\ln \Psi \left( r \right)=\frac{3}{2}\ln r+\ln f\left( {{r}^{2}}\kappa \left( r \right) \right)-\ln f\left( \kappa \left( r \right) \right).$$
  
  Derivative both sides of this equation, we have
  
  $$\frac{d\ln \Psi \left( r \right)}{dr}=\frac{3}{2r}+\frac{{f}'\left( {{r}^{2}}\kappa \left( r \right) \right)\left( 2r\kappa \left( r \right)+{{r}^{2}}{\kappa }'\left( r \right) \right)}{f\left( {{r}^{2}}\kappa \left( r \right) \right)}-\frac{{f}'\left( \kappa \left( r \right) \right){\kappa }'\left( r \right)}{f\left( \kappa \left( r \right) \right)}.$$
  
  Simplifying this, we can obtain
  
  \begin{equation*}
  	\begin{aligned}
  		\frac{{\Psi }'\left( r \right)}{\Psi \left( r \right)}& =\frac{3}{2r}+\frac{{f}'\left( {{r}^{2}}\kappa \left( r \right) \right)\left( 2r\kappa \left( r \right)+{{r}^{2}}{\kappa }'\left( r \right) \right)}{f\left( {{r}^{2}}\kappa \left( r \right) \right)}-\frac{{f}'\left( \kappa \left( r \right) \right){\kappa }'\left( r \right)}{f\left( \kappa \left( r \right) \right)} \\ 
  		& =\frac{3}{2r}+\frac{{f}'\left( {{r}^{2}}\kappa \left( r \right) \right)}{f\left( {{r}^{2}}\kappa \left( r \right) \right)}\left( 2r\kappa \left( r \right)+{{r}^{2}}{\kappa }'\left( r \right) \right)-\frac{{f}'\left( \kappa \left( r \right) \right)}{f\left( \kappa \left( r \right) \right)}{\kappa }'\left( r \right) \\ 
  		& =\frac{3}{2r}+\left( {{r}^{2}}A-B \right){\kappa }'\left( r \right)+2rA\kappa \left( r \right), \\ 
  	\end{aligned}
  \end{equation*}
  where $A:=\frac{{f}'\left( {{r}^{2}}\kappa \left( r \right) \right)}{f\left( {{r}^{2}}\kappa \left( r \right) \right)}$ and $B:=\frac{{f}'\left( \kappa \left( r \right) \right)}{f\left( \kappa \left( r \right) \right)}$.
  
  By Implicit Function Theorem and $h'(k)>0$, ${\kappa }'\left( r \right)=-\frac{\frac{\partial \phi \left( r,k \right)}{\partial r}}{\frac{\partial \phi \left( r,k \right)}{\partial k}}=-\frac{h\left( {{r}^{2}}k \right)+2{{r}^{2}}k{h}'\left( {{r}^{2}}k \right)}{{h}'\left( k \right)+{{r}^{3}}{h}'\left( {{r}^{2}}k \right)}<0$.
%  \begin{equation*}
%  	\begin{aligned}
%  		 {\kappa }'\left( r \right)&=-\frac{\frac{\partial \phi \left( r,k \right)}{\partial r}}{\frac{\partial \phi \left( r,k \right)}{\partial k}} \\ 
%  		& =-\frac{h\left( {{r}^{2}}k \right)+2{{r}^{2}}k{h}'\left( {{r}^{2}}k \right)}{{h}'\left( k \right)+{{r}^{3}}{h}'\left( {{r}^{2}}k \right)}<0.\\ 
%  	\end{aligned}
%  \end{equation*}

If ${{r}^{2}}A-B\le 0$, then for any $r\in(0,+\infty)$, ${\Psi }'\left( r \right)>0$.

If ${{r}^{2}}A-B>0$, we need further discussions.

Consider $\varphi(k)=k(h(k)^2)$, $\varphi (0)=0$ and $\underset{k\to {{1}^{-}}}{\mathop{\lim }}\,\varphi(k)=1$, ${\varphi }'\left( k \right)={{\left( h\left( k \right) \right)}^{2}}+2kh\left( k \right){h}'\left( k \right)>0$. Then there exists a unique $k_0\in (0,1)$, such that $\varphi(k_0)=\frac{1}{2}$.

When $r\to 0$, the equation (\ref{updates ziqia equation}) becomes $\kappa(r)=\frac{1}{2h(\kappa(r))^2}$, so $\underset{r\to 0}{\mathop{\lim }}\,\kappa(r)=k_0$, $\underset{r\to 0}{\mathop{\lim }}\,A=\frac{f'(0)}{f(0)}$, $\underset{r\to 0}{\mathop{\lim }}\,B=\frac{f'(k_0)}{f(k_0)}$. Then we know $\underset{r\to 0}{\mathop{\lim }}\,\frac{{\Psi }'\left( r \right)}{\Psi \left( r \right)}>0$.

When $r\to +\infty$, $\kappa(r)\to 0$. Assume $\underset{r\to +\infty}{\mathop{\lim }}\, r^2{\kappa(r)}=L$, where $L\in \mathbb{R}$ and $L$ maybe $\infty$. Then  the equation (\ref{updates ziqia equation}) becomes $\frac{L}{r^2}=\frac{1}{2r^{2}h(L)^2}$ in the sense that $r\to +\infty$, where $\underset{r\to +\infty}{\mathop{\lim }}\,h(\kappa(r))=h(0)=\frac{3}{4}$ can be ignored. %by calculating the integral and changing the variable $u={{\sin }^{2}}\theta$.
%这个可以去掉 
Due to the monotonicity of $\varphi(k)$, we know $L=k_0\in (0,1)$ is a finite number. 

So $\kappa \left( r \right)\sim \frac{k_0}{{{r}^{2}}}\left( r\to +\infty  \right)$,  $\underset{r\to +\infty}{\mathop{\lim }}\,A=\frac{f'(k_0)}{f(k_0)}$, $\underset{r\to +\infty}{\mathop{\lim }}\,B=\frac{f'(0)}{f(0)}$, ${\kappa}'\left( r \right)\sim -\frac{2k_0}{{{r}^{3}}}\left( r\to +\infty  \right)$. Then $\frac{{\Psi }'\left( r \right)}{\Psi \left( r \right)}\sim \frac{3}{2r}+\frac{2k_{0}f'(0)}{r^{3}f(0)}\left( r\to +\infty  \right)$, $\underset{r\to +\infty}{\mathop{\lim }}\,\frac{{\Psi }'\left( r \right)}{\Psi \left( r \right)}>0$. From the expression we can see $\frac{{\Psi }'\left( r \right)}{\Psi \left( r \right)}$ is continuous in $(0,+\infty)$, so for all $r\in(0,+\infty)$, $\frac{{\Psi }'\left( r \right)}{\Psi \left( r \right)}>0$, then ${\Psi}' \left( r \right)>0(r\in(0,+\infty))$.

(3)Domain and Range
  
  From $\underset{r\to 0}{\mathop{\lim }}\,\kappa(r)=k_0$ and  $\underset{r\to +\infty}{\mathop{\lim }}\,\kappa(r)=\frac{k_0}{r^2}$, we get $\underset{r\to 0}{\mathop{\lim }}\,\Psi(r)=0$ and $\underset{r\to +\infty}{\mathop{\lim }}\,\Psi(r)=+\infty$. So the domain and the range of $\Psi(r)$ are both $(0,+\infty)$.
  
In summary, we know $\Psi(r)$ is a continuous, strictly monotonically increasing bijection on $\mathbb{R}_{+}$.
  \end{proof}
  
  	Because we only need to consider periodic orbits of period 1 in problem (\ref{uncoupled ODE}), we pick $\sigma_{i}=\frac{1}{2n_i}(i=1,2)$, where $n_i \in \mathbb{Z}_{>0}(i=1,2)$ and $gcd(n_1, n_2)=1$. 
  
    Now, we establish a construction to get periodic solutions of problem (\ref{uncoupled ODE}) from solutions of problem (\ref{updates mean equation of mean interaction}). 
   
   Let 
   
   $$\bar{q}_{i}^{{{\sigma }_{i}}}\left( t \right)=\left\{ \begin{aligned}
   	& q_{i}^{{{\sigma }_{i}}}\left( t \right)\text{,  }0\le t<{{\sigma }_{i}} \\ 
   	& q_{i}^{{{\sigma }_{i}}}\left( 2{{\sigma }_{i}}-t \right)\text{,  }{{\sigma }_{i}}\le t<2{{\sigma }_{i}}. \\ 
   \end{aligned} \right.$$
   
   And we continue this extension $\bar{q}_{i}^{{{\sigma }_{i}}}\left( t\text{+2}{{\sigma }_{i}} \right)\text{=}\bar{q}_{i}^{{{\sigma }_{i}}}\left( t \right)$, then we extend the solutions of problem (\ref{updates mean equation of mean interaction}) to problem (\ref{uncoupled ODE}) periodically and uniquely.
   
   On the one hand, $\frac{\sigma_1}{\sigma_2}=\frac{n_2}{n_1}$ is a positive rational number. And since $\Psi(r):\mathbb{R}_{+}\to \mathbb{R}_{+}$ is a bijection, for any positive rational numbers, there exists a unique $r\in \mathbb{Q}_{+}$ decides $\frac{\sigma_{1}}{\sigma_{2}}$. Thereby we find the one-to-one correspondence between $\mathbb{Q}_{+}$ and  $\frac{\sigma_1}{\sigma_2}$. 
   
   On the other hand，Lemma \ref{Lemma 4.1} shows that $\frac{\sigma_1}{\sigma_2}$ gives a unique $(q_{1}^{\sigma_{1}}, q_{2}^{\sigma_{2}})$, so $\frac{\sigma_1}{\sigma_2}$ and $(\bar{q}_{1}^{{{\sigma }_{i}}},\bar{q}_{2}^{{{\sigma }_{2}}})$ are in one-to-one correspondence. According to Theorem \ref{Main1}, we can get $(\bar{q}_{1}^{{{\sigma }_{i}}},\bar{q}_{2}^{{{\sigma }_{2}}})$ and $\mathscr{C}_{\mathscr{B}_{av}}$ are in one-to-one correspondence under the  ${\mathbb{Z}}/{2\mathbb{Z}}\;\times {\mathbb{Z}}/{2\mathbb{Z}}\;$ action, where $z_i=-z_i(i=1,2)$ under the action of ${\mathbb{Z}}/{2\mathbb{Z}}\;$. 
   %On the other hand, by the construction of $z_i(i=1,2)$, we know $z_i$ change its signs at each zero, and from Theorem \ref{Main1}, $(\bar{q}_{1}^{{{\sigma }_{i}}},\bar{q}_{i}^{{{\sigma }_{2}}})$ and $(z_1,z_2)$ are in one-to-one correspondence. According to Lemma \ref{Lemma 4.1},  $\frac{\sigma_1}{\sigma_2}$ gives a unique $(q_{1}^{\sigma_{1}}, q_{2}^{\sigma_{2}})$, so $\frac{\sigma_1}{\sigma_2}$ and $(\bar{q}_{1}^{{{\sigma }_{i}}},\bar{q}_{i}^{{{\sigma }_{2}}})$ are in one-to-one correspondence.
   %这一段有问题，结论还是得在模z/2z\times z/2z意义下成立。前面的$z_i$ change its signs at each zero是由于S^1本身的拓扑结构和z是周期轨决定的，但是不能说明对特定的q我们就能决定z的符号了，还是有两个z对应的，所以还是得模去群作用。啊！我又要反思品味问题了，么得一点数学品味。
   
   Therefore, we finally get the one-to-one correspondence between the set of positive rational numbers $\mathbb{Q}_{+}$ and $\mathscr{C}_{\mathscr{B}_{av}}/ {\mathbb{Z}}/{2\mathbb{Z}}\;\times {\mathbb{Z}}/{2\mathbb{Z}}\;$.

 \end{proof}

  \section*{Acknowledgments} 
  First of all, I would like to express my sincere gratitude to Professor Urs Frauenfelde, who inspired the idea for this article. I am very grateful for his meticulous guidance and discussions, and for his inspiring work in this field. I also deeply appreciate Professor Shanzhong Sun's generous help and support, as well as the time he devoted to discussing and verifying the details with me for multiple times. Finally, my heartfelt thanks go to Professor Lei Zhao, for attending my presentation and offering constructive suggestions. Partially supported by NSFC(No.12171327).
  %  \subsection*{4.Instantaneous Interaction}
   % The instantaneous interaction motion of equations is
    % \begin{equation}
    %	\begin{aligned}
    %		&\ddot{q_1}=\frac{2}{\left(q_1\right)^2}-\frac{1}{ \left(\bar{q}_{2}-\bar{q}_{1}\right)^2},\\
    %		&\ddot{q_2}=-\frac{2}{\left(q_2\right)^2}+\frac{1}{ \left(\bar{q}_{2}-\bar{q}_{1}\right)^2},
    %	\end{aligned}
    %\end{equation}
    %where $q_1<0<q_2$.
    
%	Basically, the two electrons act similiarly in this configuration(i.e. $\left| q_1\right|=q_2$), so we can just consider one of them, the motion of the other is exactly the same.
	
%	\subsection*{2.Variational approach}
%	1.Use the basic konwledge of celestial mechanics, as un analoge of n-body problem, such as regularization.
	
%	2.Establish action functional and find the critical point(i.e the period orbits), construct loop space and use floer homology.
	
%	3.Classify the periodic orbits.
	
%	\subsection*{3.Quantum chaos}
%	Semicalssical analysis and quantization.
	
%	\subsection*{4.Miscellanea}
%	Compare the relationship between Helium and hygedron aotms. Some open problems or some unsolved difficulties or some interesting examples and other things. 
	
%	插入参考文献这样%\cite{Kac1966CanOH}
	\clearpage
	%\phantomsection
	\addcontentsline{toc}{section}{Reference}
	\tolerance=500
	\bibliographystyle{plain}
	\bibliography{referenceh}

@article{wintgen1992semiclassical,
	title={The semiclassical {H}elium atom},
	author={Wintgen, Dieter and Richter, Klaus and Tanner, Gregor},
	journal={Chaos: An Interdisciplinary Journal of Nonlinear Science},
	volume={2},
	number={1},
	pages={19--33},
	year={1992},
	publisher={American Institute of Physics}
}

@article{barutello2021regularized,
	title={Regularized variational principles for the perturbed {K}epler problem},
	author={Barutello, Vivina and Ortega, Rafael and Verzini, Gianmaria},
	journal={Advances in Mathematics},
	volume={383},
	pages={107694},
	year={2021},
	publisher={Elsevier}
}

@article{cieliebak2022variational,
	title={A variational approach to frozen planet orbits in {H}elium},
	author={Cieliebak, Kai and Frauenfelder, Urs and Volkov, Evgeny},
	journal={Annales de l'Institut Henri Poincar{\'e} C},
	volume={40},
	number={2},
	pages={379--455},
	year={2022}
}

@article{tanner2000theory,
	title={The theory of two-electron atoms: between ground state and complete fragmentation},
	author={Tanner, Gregor and Richter, Klaus and Rost, Jan-Michael},
	journal={Reviews of Modern Physics},
	volume={72},
	number={2},
	pages={497},
	year={2000},
	publisher={APS}
}

@article{baranzini2025frozen,
	title={Frozen Planet Orbits for the n-Electron Atom: S. Baranzini, GM Canneori and S. Terracini},
	author={Baranzini, Stefano and Canneori, Gian Marco and Terracini, Susanna},
	journal={Communications in Mathematical Physics},
	volume={406},
	number={12},
	pages={290},
	year={2025},
	publisher={Springer}
}

@article{zhao2023shooting,
	title={Shooting for collinear periodic orbits in the {H}elium model},
	author={Zhao, Lei},
	journal={Zeitschrift f{\"u}r Angewandte Mathematik und Physik},
	volume={74},
	number={6},
	pages={227},
	year={2023},
	publisher={Springer}
}

@article{frauenfelder2023compactness,
	title={A compactness theorem for frozen planets},
	author={Frauenfelder, Urs},
	journal={Journal of Topology and Analysis},
	volume={15},
	number={02},
	pages={527--543},
	year={2023},
	publisher={World Scientific}
}

@article{cieliebak2023nondegeneracy,
	title={Nondegeneracy and integral count of frozen planet orbits in {H}elium},
	author={Cieliebak, Kai and Frauenfelder, Urs and Volkov, Evgeny},
	journal={Tunisian Journal of Mathematics},
	volume={5},
	number={4},
	pages={713--770},
	year={2023},
	publisher={Mathematical Sciences Publishers}
}

@article{frauenfelder2021helium,
	title={Helium and {H}amiltonian delay equations},
	author={Frauenfelder, Urs},
	journal={Israel Journal of Mathematics},
	volume={246},
	number={1},
	pages={239--260},
	year={2021},
	publisher={Springer}
}

@article{ortega2011linear,
	title={Linear motions in a periodically forced {K}epler problem},
	author={Ortega, Rafael},
	journal={Portugaliae Mathematica},
	volume={68},
	number={2},
	pages={149--176},
	year={2011}
}

@inproceedings{bahri1991periodic,
	title={Periodic solutions of {H}amiltonian systems of 3-body type},
	author={Bahri, Abbas and Rabinowitz, Paul Henry},
	booktitle={Annales de l'Institut Henri Poincar{\'e} C, Analyse Non Lin{\'e}aire},
	volume={8},
	number={6},
	pages={561--649},
	year={1991},
	organization={Elsevier}
}

@article{ambrosetti1990closed,
	title={Closed orbits of fixed energy for singular {H}amiltonian systems},
	author={Ambrosetti, Antonio and Zelati, Vittorio Coti},
	journal={Archive for Rational Mechanics and Analysis},
	volume={112},
	number={4},
	pages={339--362},
	year={1990},
	publisher={Springer}
}

@article{tanaka1993prescribed,
title={A prescribed energy problem for a singular {H}amiltonian system with a weak force},
author={Tanaka, Kazunaga},
journal={Journal of Functional Analysis},
volume={113},
number={2},
pages={351--390},
year={1993},
publisher={Elsevier}
}

@article{rabinowitz1994note,
	title={A note on periodic solutions of prescribed energy for singular {H}amiltonian systems},
	author={Rabinowitz, Paul Henry},
	journal={Journal of Computational and Applied Mathematics},
	volume={52},
	number={1-3},
	pages={147--154},
	year={1994},
	publisher={Elsevier}
}

@article{coti1994collision,
	title={Collision and non-collision solutions for a class of {K}eplerian-like dynamical systems},
	author={Coti Zelati, Vittorio and Serra, Enrico},
	journal={Annali di Matematica Pura ed Applicata},
	volume={166},
	number={1},
	pages={343--362},
	year={1994},
	publisher={Springer}
}

@article{rebelo2018periodic,
	title={PERIODIC LINEAR MOTIONS WITH MULTIPLE COLLISIONS IN A FORCED {K}EPLER TYPE PROBLEM.},
	author={Rebelo, Carlota and Sim{\~o}es, Alexandre},
	journal={Discrete and Continuous Dynamical Systems: Series A},
	volume={38},
	number={8},
	year={2018}
}

@article{zhao2016some,
	title={Some collision solutions of the rectilinear periodically forced {K}epler problem},
	author={Zhao, Lei},
	journal={Advanced Nonlinear Studies},
	volume={16},
	number={1},
	pages={45--49},
	year={2016},
	publisher={De Gruyter}
}

@article{barutello2008singularities,
	title={On the singularities of generalized solutions to n-body-type problems},
	author={Barutello, Vivina and Ferrario, Davide L and Terracini, Susanna},
	journal={International Mathematics Research Notices},
	volume={2008},
	number={9},
	pages={rnn069--rnn069},
	year={2008},
	publisher={OUP}
}
	
\end{document}